\documentclass[12pt]{amsart}
\usepackage{latexsym, amsmath, amssymb,amsthm,amsopn,amsfonts,amsxtra}
\usepackage{epsfig,graphics,color,graphicx,graphpap}
\usepackage[colorinlistoftodos]{todonotes}

\graphicspath{{Figures/}}

\newcommand{\Tper}{T_{\rm period}}
 
\newcommand{\CC}{{\mathbb C}}

\newcommand{\cN}{{\mathcal N}}
\newcommand{\cNcr}{{\mathcal N}_{\rm cr}}

\newcommand{\RR}{{\mathbb R}}

\newcommand{\Hduff}{H_{\rm Duffing}}
\newcommand{\HDW}{H_{\rm DW}}

\newcommand{\W}{\Omega}
\newcommand{\w}{\omega}

\theoremstyle{plain}

\newcommand{\abs}[1]{\left\lvert#1\right\rvert}
\newcommand{\norm}[1]{\left\lVert#1\right\rVert}

\newcommand{\nn}{\nonumber}
\newcommand{\dd}{\sigma}
\newcommand{\cn}{\operatorname{cn}}
\newcommand{\dn}{\operatorname{dn}}
\newcommand{\sn}{\operatorname{sn}}
\newcommand{\NcrFD}{N_{\rm cr}}

\newtheorem{thm}{Theorem}[section]
\newtheorem{prop}{Proposition}[section]

\newtheorem{lem}[prop]{Lemma}

\theoremstyle{definition}

\newtheorem{rem}{Remark}[section]

\numberwithin{equation}{section}

\def\squarebox#1{\hbox to #1{\hfill\vbox to #1{\vfill}}}

\newcommand{\ren}{{\hspace{1pt}{\rm ren}}}

 \definecolor{skyblue}{rgb}{0.85,0.85,1}

    \newcommand{\bigO}[1]{\ensuremath{\mathop{}\mathopen{}\mathcal{O}\mathopen{}\left(#1\right)}}

\title[Self-trapping and Josephson tunneling solutions to NLS]
{Self-trapping and Josephson tunneling solutions to the nonlinear Schr\"odinger / Gross-Pitaevskii Equation}  

\author[R.H. Goodman]
{Roy H. Goodman}
\email{goodman@njit.edu}

\address{Department of Mathematical Sciences, New Jersey Institute of Technology \\
University Heights, Newark, NJ  07102, USA}

\author[J.L. Marzuola]
{Jeremy L. Marzuola}
\email{marzuola@math.unc.edu}

\address{Mathematics Department, University of North Carolina, \\
Phillips Hall, CB\#3250, Chapel Hill, NC 27599, USA}

\author[M. I. Weinstein]
{Michael I. Weinstein}
\email{miw2103@columbia.edu}

\address{Department of Applied Physics and Applied Mathematics, Columbia University \\
200 S. W. Mudd, 500 W. 120th St., New York City, NY 10027, USA}

\begin{document}    

\begin{abstract}
We study the long-time behavior of solutions to the nonlinear Schr\"odinger / Gross-Pitaevskii equation (NLS/GP) with a symmetric double-well potential.  NLS/GP governs nearly-monochromatic guided optical beams in weakly coupled waveguides with both linear and nonlinear (Kerr) refractive indices and zero absorption, as well as the behavior of Bose-Einstein condensates.  For small $L^2$ norm (low power), the solution executes beating oscillations between the two wells. There is a power threshold at which a symmetry breaking bifurcation occurs.  The set of guided mode solutions splits into two families of solutions. One type of solution is concentrated in either well of the potential, but not both. Solutions in the second family undergo tunneling oscillations between the two wells.  A finite dimensional reduction (system of ODEs) derived in~\cite{MW} is expected to well-approximate the PDE dynamics on long time scales. In particular, we revisit this reduction, find a class of exact solutions and shadow them in the (NLS/GP) system by applying the approach of~\cite{MW}.
\end{abstract}

\date{\today}
\maketitle 

\section{Introduction}

We study the long-time behavior of solutions to the nonlinear Schr\"odinger / Gross-Pitaevskii equation (NLS/GP) with a symmetric double-well potential in $\RR^1$.  Equations of NLS/GP-type, in one or more spatial dimensions, arise as models of many physical systems, notably (a) nearly-monochromatic guided optical beams in weakly coupled waveguides with both linear and nonlinear (Kerr) refractive indices and no absorption~\cite{Boyd:2008,Newell:2003}, and (b) the time-evolution of a Bose-Einstein condensate (BECs) confined by a magnetically-induced linear potential~\cite{Chen1,Chen2,HolmerChen,ESY,PS:03}.

Specifically, we consider NLS/GP in one space dimension
\begin{equation}
\label{eqn:nlsdwp-gK}
i \partial_t u(x,t) = \left(- \partial_x^2 + V_\ell(x)\right) u(x,t)\ +\ g  \abs{u(x,t)}^2 u(x,t), \; u(x,t) : \RR \times \RR \to \CC, 
\end{equation}
where   $V_\ell(x)$ denotes a double well potential. In the context of optics, $u(x,t)$ denotes the complex-valued slowly varying envelope of the electric field for a nearly monochromatic stationary beam propagating in the ``$t$'' direction. The potential $V_\ell(x)$ is obtained from the transverse refractive index  profile of the two coupled waveguides and $g<0$  is proportional to the Kerr nonlinear coefficient. In the macroscopic quantum setting, $u(x,t)$ denotes a macroscopic (mean-field) wave-function, $V_\ell$, the magnetic trap and $g$, which is proportional  to the microscopic two-body scattering length, can be positive or negative, depending on the underlying atomic species.

A key  quantity conserved by solutions of~\eqref{eqn:nlsdwp-gK} is the squared $L^2$ norm, 
\begin{equation}
\cN \ \equiv \  \int_{-\infty}^{\infty} |u(x,t)|^2 dx =  \int_{-\infty}^{\infty} |u(x,0)|^2 dx 
\label{powerN}
\end{equation}
corresponding in the two systems to (a) the optical power, conserved with propagation distance, and (b) the particle number, conserved under the time evolution. The aim of this paper is to describe how long term dynamics of specific solutions of equation~\eqref{eqn:nlsdwp-gK} vary with $\cN$, extending the result of~\cite{MW}.

\subsection{Background and prior results}

\subsubsection*{Linear theory of double wells}
We take the double-well, $V_\ell(x)$,  to be bimodal, i.e.\ the sum of translates of a unimodal potential $V_0(x)$,
\begin{equation*}
V_\ell(x) = V_0 (x-\ell) + V_0 (x+\ell),\; \ell>0.
\end{equation*}
We shall assume that the basic well,  $V_0(x)$, is spatially localized and even with respect to $x=0$:
\[ V_0(x) = V_0(-x),\]
and supports exactly one discrete eigenpair, $(\Omega_\star,\varphi_\star(x))$ which solves
 \begin{equation*}
 \left(-\partial_x^2+V_0(x)\right)\varphi_\star\ =\ \Omega_\star\varphi_\star.
 \end{equation*}
Without loss of generality, we assume $\|  \varphi_\star \|_{L^2} =1$.

For $\ell$ large and positive, detailed information  for the double-well Schr\"odinger operator  
\begin{equation*}
H_\ell = -\partial_x^2 + V_\ell.
\end{equation*}
can be deduced from the properties of the basic single well, $V_0(x)$~\cite{Har:80}. 
In particular, there is a well-separation distance, $\ell_0>0$, such that  if  $\ell>\ell_0$ (weak-coupling), then the linear operator $H_\ell$ has  two simple eigenvalues, $\W_0 = \W_0 (\ell)$ and $\W_1 = \W_1 (\ell)$, with 
\begin{equation*}
\W_0(\ell) \ <\;W_\star\ <\;W_1(\ell) \text{and } \W_1(\ell)-\W_0(\ell)\ =\ 
\bigO{e^{-c_0\ell}},
\end{equation*}
for some $c_0>0$.
By the symmetry of $V_\ell$, the corresponding eigenfunctions  $\psi_j=\psi_j(x;\ell)$, satisfying
$$
H_\ell \psi_j = \W_j \psi_j, \ j = 0,1
$$
may be taken to be  $L^2$-normalized and possess, respectively, even and odd spatial symmetry:
\begin{equation*}
\psi_0(x)=\psi_0(-x),\; \psi_1(x)=-\psi_1(-x).
\end{equation*}
For sufficiently large $\ell$, these eigenfunctions are bimodal and satisfy
\begin{equation}
\psi_j(x;\ell)\; \approx\frac{1}{\sqrt{2}}\left(\varphi_\star(x-\ell)\ +\ {(-1)}^j \varphi_\star(x+\ell)\ \right), \; \ell\gg1, \; j=0,1.
\label{varphi-approx}
\end{equation}

\subsubsection*{Nonlinear standing waves}

There has been a great deal of interest in the existence, stability, and bifurcations of stationary solutions to NLS/GP~\eqref{eqn:nlsdwp-gK} of standing wave type:
\begin{equation*}
u(x,t) = U(x) e^{-i \omega t}.
\end{equation*}
For fixed value $\cN>0$, $U=U(x;\cN)$ and $\w = \w(\cN)$ denotes a solution of the nonlinear eigenvalue problem defined jointly by equation~\eqref{powerN} and by
\begin{equation}
\left(- \partial_x^2 + V_\ell(x)\right) U\ + g |U|^2 U =\omega U  ,\ U\in H^1(\mathbb{R}).\label{eqn:stationary}
\end{equation}

Questions of both mathematical and physical interest are:
\begin{enumerate}
\item Given $\cN>0$, classify the  solutions of~\eqref{eqn:stationary}.
\item  How does the set of solutions vary as $\cN$ is increased?
\end{enumerate}

In~\cite{KKSW}, it is shown that there exist solution branches which are continuations of the linear solutions as $\cN\to0$, i.e.\ such that for $j=0,1$,  $U_j(x;\cN) \sim \sqrt{\cN} \psi_j(x)$ and $\w \to \W_j$, and that these branches possess the same symmetries as their linear counterparts. In the case of focusing nonlinearity ($g<0$), they find that the ``ground state solution,'' the continuation of the linear ground state $\psi_0(x)$,  undergoes  a symmetry-breaking bifurcation as $\cN$ is increased. Fix $\ell>\ell_1$ sufficiently large.  Then, there exists a threshold power / particle number, $\cNcr(\ell)$, such that for $0<\cN<\cNcr(\ell)$, the only solutions to equation~\eqref{eqn:stationary} are $U_0(x;\cN)$ and $U_1(x;\cN)$, but for $\cN> \cNcr$, there exists another pair of solutions $U_{\pm}(x;\cN)$, concentrated in the left or in the right well but not symmetrically in both. As $\cN$ is increased further, the solution becomes concentrated more strongly in one well or the other. 
Moreover, for $\cN>\cNcr$, stability is transferred from the ground state $U_0(x;\cN)$ to the asymmetric states $U_\pm(x;\cN)$.

The analysis of~\cite{KKSW} is based on a Lyapunov-Schmidt reduction. 
For $\cN>0$ small, solutions are decomposed into their projections onto the span of $\{\psi_0,\psi_1\}$ and their projection onto the orthogonal complement. For $\ell$ large, one may view this system as a two-dimensional system of nonlinear algebraic equations, which is weakly coupled to an infinite dimensional system. The two-dimensional truncated system has a \emph{bifurcation diagram} of the type described above and it can be shown that neglected (infinite-dimensional) corrections are small for $\ell$ large. In particular, if $N_{\rm cr}(\ell)$ denotes the approximate symmetry breaking threshold obtained from the 2-dimensional reduction, then we have that 
$$ |N_{\rm cr}(\ell) - \cNcr(\ell)| \leq \mathcal{O} ( |\Omega_0 (\ell) - \Omega_1 (\ell) |) \approx e^{- \kappa \ell}, \ \text{for some} \ \kappa > 0,$$
see~\cite{KKSW}, Equation $(1.1)$.

 \begin{rem}
Throughout this paper we shall assume $g<0$, the case of focusing nonlinearity. For the defocusing nonlinearity $g>0$, an analogous result holds with one difference. It is the mode $U_1 (x,\cN)$ that undergoes a symmetry-breaking bifurcation as $\cN$ is increased; see for example~\cite{TKFS}.
\end{rem}

The standing-wave result has been generalized in several ways. Kirr, Kevrekidis, and Pelinovsky~\cite{KKP} perform a global bifurcation analysis for the class of symmetric double well potentials with a non-degenerate maximum.  
Yang, in~\cite{Yang:2012,Yang:2012to},  studies the detection and classification of symmetry-breaking bifurcations in NLS equations with more general nonlinearities.
In~\cite{KKC}, Kapitula et al.\ study the richer family of bifurcations for NLS with triple well potentials.


\subsubsection*{Time-dependent dynamics}
To study the dynamics of solutions to NLS/GP near the symmetry breaking bifurcation it is natural to express the solution to the initial value problem as
\begin{equation}
u(x,t) = c_0(t) \psi_0(x) + c_1(t) \psi_1(x) + \rho(x,t),\; \left\langle \psi_j(\cdot),\rho(\cdot,t)\right\rangle=0,\ j=1,2,
\label{eqn:ansatz}
\end{equation}
where $c_0(t)$ and $c_1(t)$ are time-dependent complex-valued amplitudes and $\rho(x,t)$, the projection of $u(x,t)$ onto $\operatorname{span}^\perp\{\psi_0,\psi_1\}$. The resulting system consists of ODEs for $c_0(t)$ and $c_1(t)$ coupled to a PDE solved by $\rho(x,t)$ and is mathematically equivalent to NLS/GP.

In~\cite{MW}, the latter two authors pursue the following strategy. First, they derive a finite-dimensional model by neglecting the terms involving $\rho(x,t)$ in the evolution equations for $c_0(t)$ and $c_1(t)$. This is reduced by symmetry to an ODE resembling the Duffing equation, which is studied by phase-plane methods. Second, they show that solutions to the full PDE system shadow the ODE solutions on very long time-scales.

The standing wave solutions $U_j(x;\cN)$ and $U_\pm(x;\cN)$ correspond to fixed points of this reduced system. The stability of the standing waves corresponds to that of these fixed points, as shown in~\cite{KKSW}. Since the reduced system is conservative and two-dimensional, stable fixed points are surrounded by families of nested periodic orbits. The result proven in~\cite{MW} is that the periodic orbits sufficiently close to these stable fixed points are shadowed by solutions to~\eqref{eqn:nlsdwp-gK} over long but finite times. {\it The result of the present paper is to extend the shadowing theorem to other periodic solutions of the ODE that are not confined to small neighborhoods of stable fixed points.}\@ This is explained further using the phase portrait in Section~\ref{sec:qualitative} below.

In recent related work, Pelinovsky and Phan~\cite{PP} use a different reduction ansatz that leads to the standard Duffing oscillator. 
They control shadowing of a wider class of orbits than in~\cite{MW} using only energy-type estimates and Gronwall's inequality in large data and arbitrary~$\ell$ regimes provided there exist two distinct eigenvalues for~$H_\ell$.  On the other hand, in the decomposition~\eqref{eqn:ansatz}, the representation of the initial conditions is more straightforward.   In the present work, we furthermore extend the results of~\cite{MW} to include orbits outside the separatrix, expanding the class of orbits we can shadow to be much closer to that of~\cite{PP}.

\subsection{Qualitative discussion of results}
\label{sec:qualitative}
\subsubsection*{The finite-dimensional model}
In~\cite{MW}, by neglecting coupling to $\rho(t)$, equation~\eqref{eqn:nlsdwp-gK} is viewed as a perturbation of the following two degree of freedom Hamiltonian ODE system for the evolution of $\left(c_0(t),c_1(t)\right)$. Under the change of variables~\cite{KW:1992,SR-K:2005}
\begin{equation}
\label{eqn:c0c1}
\begin{split}
c_0(t) &= A(t) e^{i \vartheta(t)}, \\
c_1(t) & = (\alpha(t) + i \beta(t)) e^{i \vartheta(t)},
\end{split}
\end{equation}
the evolution of the overall phase $\vartheta(t)$ decouples from the evolution of the other three quantities, due to the phase invariance of the underlying physical system, to give
\begin{subequations}
 \begin{equation}
\label{eqn:dy}
\left\{
\begin{aligned}
\dot{\alpha} &= \left(\Omega_{10}+2\alpha^2\right)\, \beta , \\
\dot{\beta} &= -\left(\Omega_{10}+2\alpha^2-2A^2\right) \alpha , \\
\dot{A} &= -2 \alpha \beta A ,
\end{aligned}
\right.
\end{equation}
and
 \begin{equation}
\dot{\vartheta} = \omega - \Omega_0 + A^2+3\alpha^2+\beta^2,
\label{varthetadot}
\end{equation}
\label{syst-rot}
\end{subequations} 
for
 \begin{equation}
\Omega_{10}= \Omega_1 - \Omega_0.
\label{omega10}
\end{equation}
We can assume $A\ge0$ with $A=0$ only on the invariant circle $\alpha^2+\beta^2=N$.

As a shorthand for these coordinates, we define a four-dimensional vector and its three-dimensional truncation 
\begin{equation*}
\chi(t) = (\alpha,\beta,A,\vartheta) \; \text{and } \tilde\chi(t) = (\alpha,\beta,A).
\end{equation*} 
A direct consequence of~\eqref{eqn:dy} is the conserved quantity
\begin{equation}
N =A^2 + \alpha^2 + \beta^2= \abs{c_0}^2 + \abs{c_1}^2,
\label{N}
\end{equation} 
corresponding to the $L^2$ invariance of~\eqref{eqn:nlsdwp-gK}.
The constraint~\eqref{N} can be used to further reduce~\eqref{eqn:dy} to the $2$-dimensional system
\begin{equation}
\label{eqn:alphabeta}
\frac{d}{dt}
\begin{bmatrix}
\alpha \\
\beta \end{bmatrix}
=  \begin{bmatrix}
\W_{10} \beta \\
(2N-\Omega_{10}) \alpha 
\end{bmatrix}  + 
\begin{bmatrix}
2 \alpha^2 \beta \\
-4 \alpha^3 - 2 \alpha \beta^2  
\end{bmatrix} = J \nabla \HDW,
\end{equation}
where 
\begin{equation}
\HDW(\alpha,\beta) = \frac{\W_{10}}{2}\beta^2 + \frac{\sigma}{2} \alpha^2 +\alpha^4 + \alpha^2 \beta^2 \ \text{and} \ J  = \begin{bmatrix} 0&1\\-1&0 \end{bmatrix}
\label{HDW}
\end{equation}
 for 
 \begin{equation}
\Omega_{10}= \Omega_1 - \Omega_0, \; 
\sigma=\Omega_{10}-2 N.
\label{sigma}
\end{equation}
Note that~\eqref{HDW} is a family of Hamiltonian systems parametrized by $N$.

The Hamiltonian $\HDW$ differs from the standard Duffing Hamiltonian,
\begin{equation}
\Hduff =\frac{1}{2}p^2 + \frac{\sigma}{2} q^2 + \frac{1}{4}q^4,
\label{eqn:duffing}
\end{equation} 
by the presence of the mixed term $\alpha^2\beta^2$.


The phase plane of system~\eqref{eqn:alphabeta} is displayed in Figure~\ref{pitchfork}.  It is topologically equivalent to that of the Duffing system~\eqref{eqn:duffing} for either sign of $\sigma$. For  $\sigma<0$, the phase plane is foliated by concentric closed curves enclosing the fixed point at the  origin; see Figure~\ref{pitchfork}(a). For $\sigma>0$, the origin becomes unstable and two stable fixed points appear, one on either side of the origin; Figure~\ref{pitchfork}(b).  In this regime, the dynamics feature two types of periodic orbits. There are two families of smaller periodic orbits, encircling exactly one of the stable fixed points and a second family of periodic orbits encircling all three fixed points.

\begin{figure}[htbp] 
   \centering
   \includegraphics[width=3in]{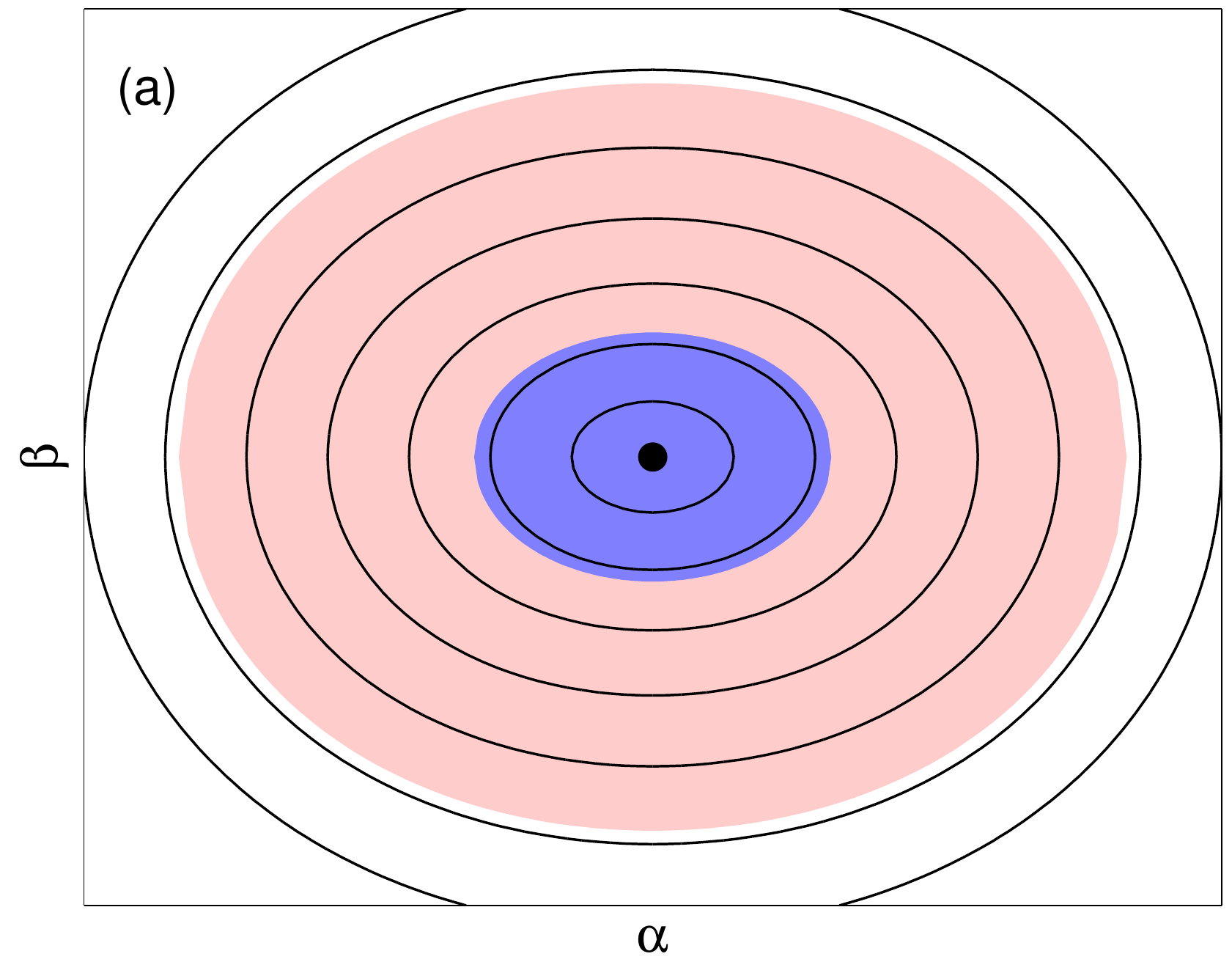} 
      \includegraphics[width=3in]{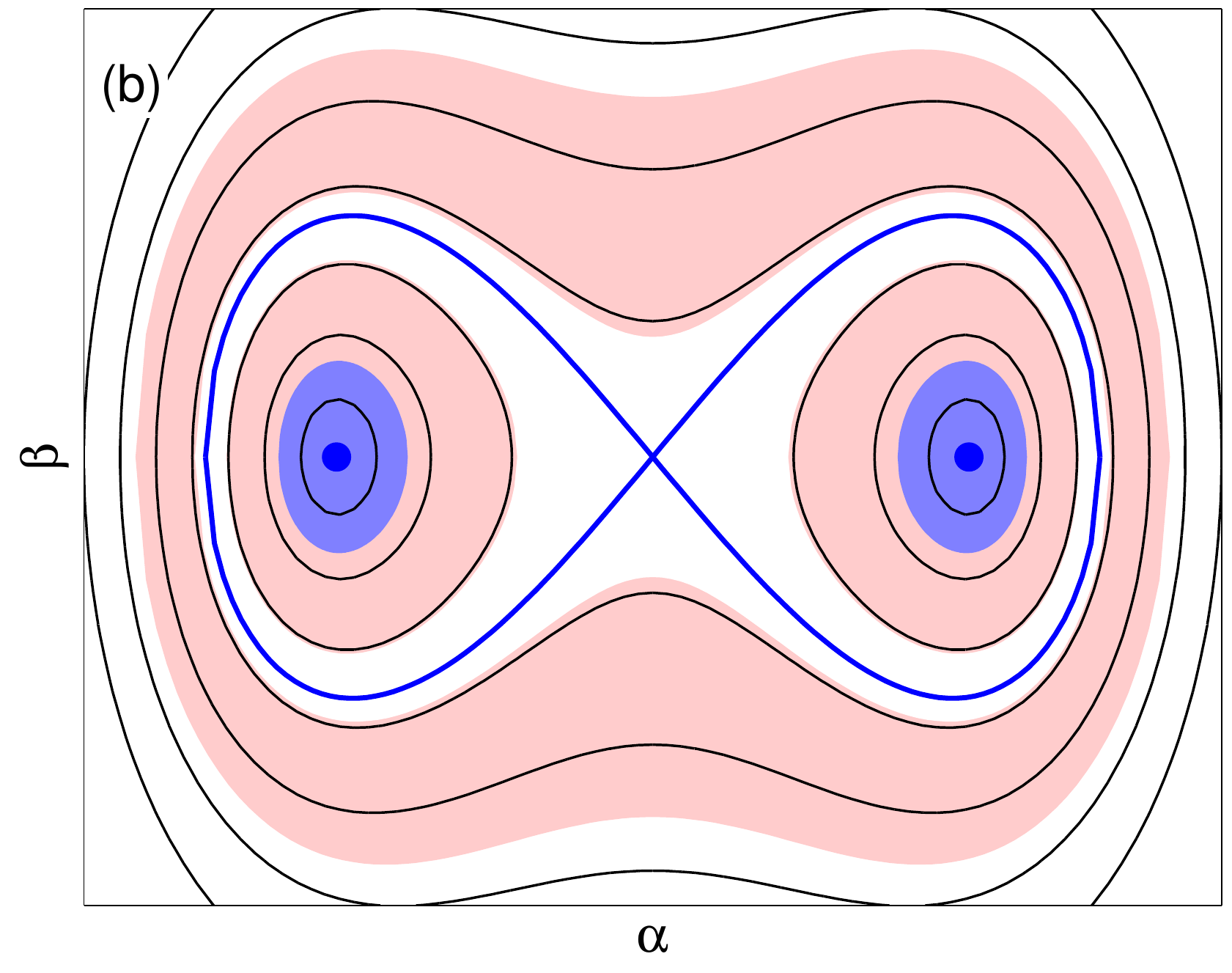} 
   \caption{The phase plane of equation~\eqref{eqn:alphabeta} with (a) $N<\NcrFD$ and (b) $N>\NcrFD$. The blue (darker) shaded regions represent the domain of validity of the proof in~\cite{MW} and the pink (lighter) shaded regions, the domains of validity in this paper.}
\label{pitchfork}
\end{figure}

Control of large time behavior for solutions to~\eqref{eqn:nlsdwp-gK} requires precise estimates of the period and amplitude of the periodic orbits in the reduced dynamical system.  These estimates are straightforward for separable Hamiltonians of the form 
$$H = p^2 + V(q),$$ 
such as the Duffing oscillator.  Although $\HDW$ is not separable, we nevertheless obtain closed-form periodic orbits in Section~\ref{s:exact}.  These are used to obtain appropriate estimates on the periods and amplitudes required for the shadowing analysis in Section~\ref{sec:rescaled}.

\subsubsection*{Interpretation of the ODE solutions}

Our results describe orbits which have an interpretation in quantum and electromagnetic contexts.   In the electromagnetic context, the orbits we study represent nearly monochromatic beams within neighboring wave-guides exchanging energy.  

For quantum settings, by examining the form of the eigenfunctions~\eqref{varphi-approx}, the solution anszatz~\eqref{eqn:ansatz}, and the reduction~\eqref{eqn:c0c1}, notice that when $\alpha>0$ and $\beta=0$, the basis functions~$\psi_0$ and $\psi_1$ interfere constructively in the left well, centered at $x=-L$ and destructively in the right well, centered at $x=L$, so that most of the optical power is concentrated in the left well. When $\alpha<0$, this is reversed and most of the power is concentrated in the right well. Thus, for the two families of periodic orbits encircling only one of the two fixed points in figure~\ref{pitchfork}(b), most energy stays in one or the other of the two potential wells. Solutions of this type are called  \textbf{self-trapped} in the Bose-Einstein condensate literature~\cite{Albiez:2005}.

For the two families of periodic orbits in which $\alpha(t)$ changes sign, both for $N<\NcrFD$ and for solutions outside the separatrix for $N>\NcrFD$, the location and magnitudes of the maximum and minimum of the reconstructed solution alternate with a fixed period. In the subcritical case $N<\NcrFD$ of Figure~\ref{pitchfork}(a), the phase difference between $c_0$ and $c_1$ changes slowly due to the closeness of the frequencies $\W_0$ and $\W_1$. For small-amplitude periodic orbits, this is a manifestation of the common beating phenomenon, while for solutions further from the origin, we refer to this as \textbf{nonlinear beating}. Periodic orbits encircling all three fixed points in Figure~\ref{pitchfork}(b) are known as \textbf{Josephson tunneling} solutions in the BEC literature, where the word tunneling refers to the fact that these orbits must cross a local maximum of the potential energy to travel between the two wells. Note that the nonlinear beating solutions speed up as they cross $\alpha=0$, while the Josephson tunneling solutions slow down. At large amplitude, there is no practical difference between Josephson tunneling and nonlinear beating solutions.  Both self-trapped and Josephson tunneling solutions have been directly observed in Bose-Einstein condensates by Albiez et al.~\cite{Albiez:2005}. 

It is instructive, also, to construct approximate periodic solutions using equations~\eqref{eqn:ansatz} and~\eqref{eqn:c0c1}
corresponding to the different families of trajectories displayed in Figure~\ref{pitchfork}. These are shown in Figures~\ref{reconstructed_sub} and~\ref{reconstructed_super}.
Figure~\ref{reconstructed_sub} shows two nonlinear beating solutions  in the subcritical regime $N< \NcrFD$. The first is a small periodic orbit around to the origin, in which case the solution stays very close to the symmetric mode $\psi_0$. In the second, the values of $A$ and $\abs{\alpha(t)+i \beta(t)}$ are close, and the solution migrates almost completely from the right well to the left well each period. 

\begin{figure}[htbp] 
   \centering
   \includegraphics[width=.45\textwidth]{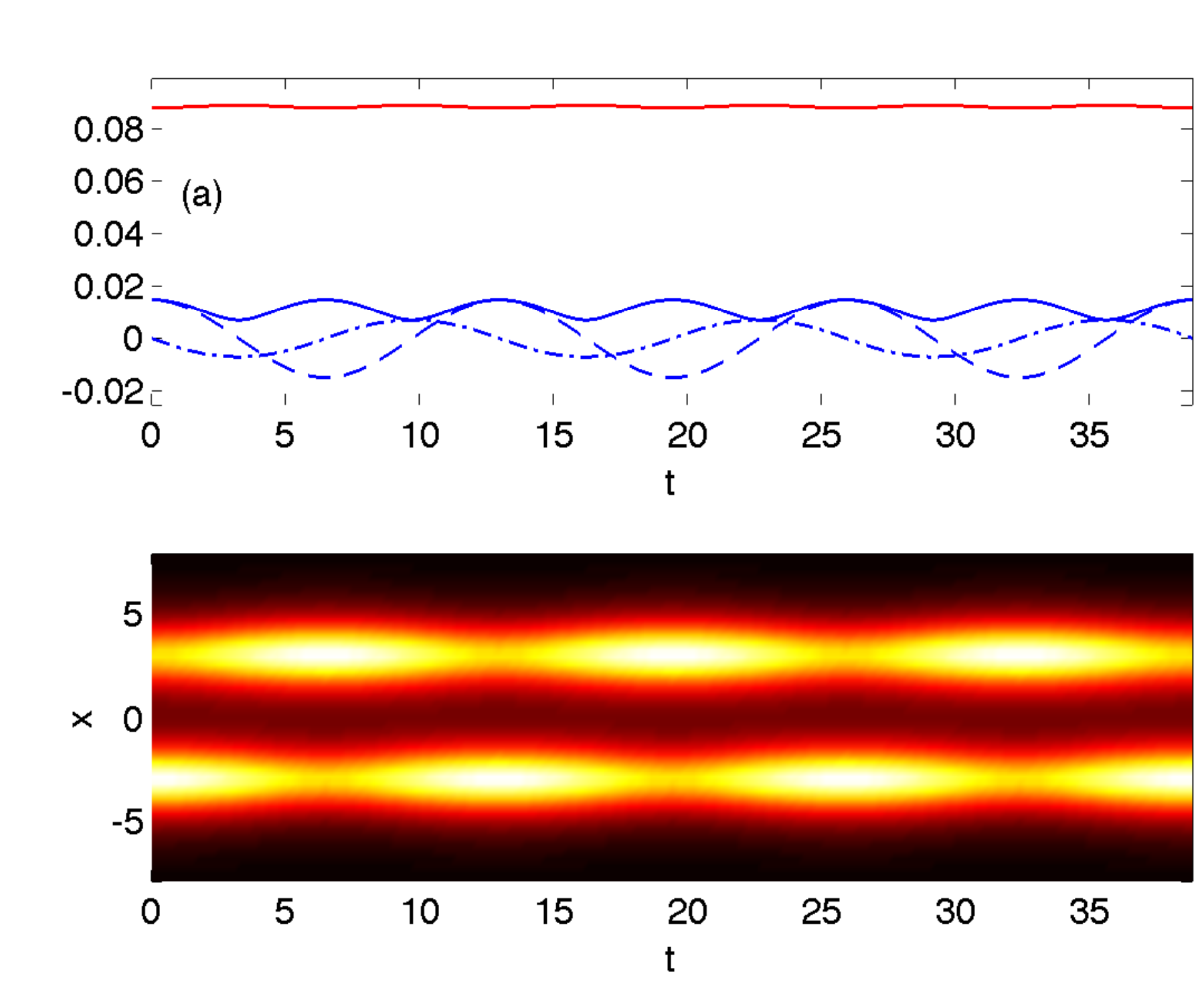}
   \includegraphics[width=.45\textwidth]{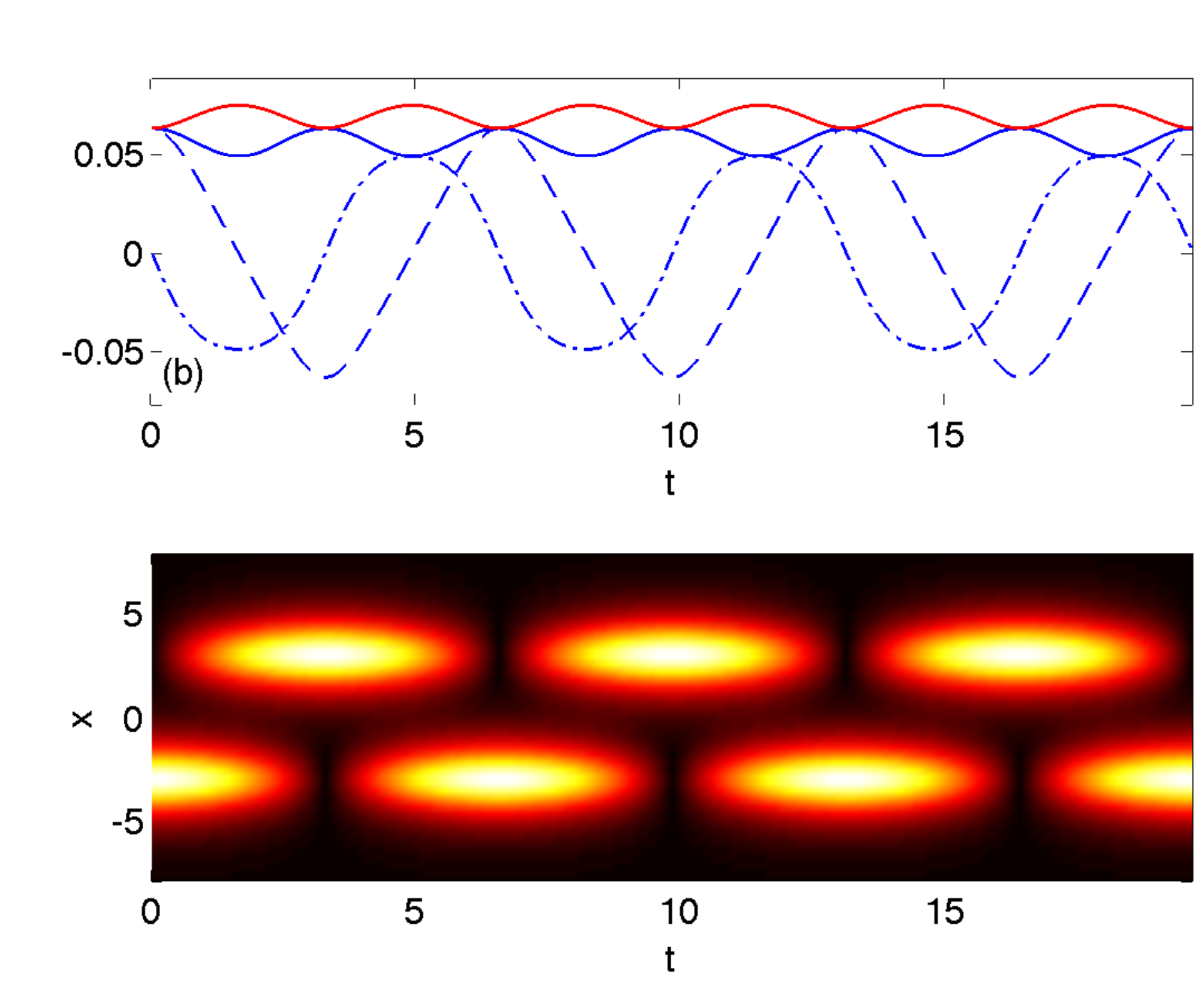}
   \caption{Nonlinear beating solutions of system~\eqref{eqn:alphabeta} in the case $N<\NcrFD$. Top row: $A(t)$ (red), $\abs{\alpha(t)+i\beta(t)}$ (blue solid), $\alpha(t)$ (blue dashed), $\beta(t)$ (blue dash-dot).  Bottom row: absolute value of reconstructed field, $c_0 \psi_0 + c_1 \psi_1$, in~\eqref{eqn:ansatz}. Subfigure~(a) shows a solution near the stable fixed point, and~(b) shows a larger periodic orbit.}
\label{reconstructed_sub}
\end{figure}

Figure~\ref{reconstructed_super} shows four orbits in the supercritical regime $N>\NcrFD$.  The first two are self-trapped, and the last two are Josephson tunneling solutions.  Subfigure~\textbf{(a)} shows a solution close to the right fixed point, the field makes small oscillations about a steady asymmetric profile. Subfigures~\textbf{(b)} and~\textbf{(c)} show solution trajectories near the separatrix. Both these trajectories spend long times close to the hyperbolic fixed point $\alpha=\beta=0$, where the field is nearly symmetric, with short bursts to asymmetric states. In~\textbf{(b)}, inside the separatrix, all these bursts move toward one asymmetric state, but in~\textbf{(c)}, the field alternates between the two. Finally in subfigure~\textbf{(d)}, the solution makes larger swings between the two asymmetric states without pausing near the symmetric state. The earlier result of~\cite{MW} essentially shows the existence of the solutions of type~\textbf{(a)} in both the subcritical and supercritical regimes, while the present result allows for patterns similar to those seen in the rest of the figures. Direct numerical simulations of NLS/GP can be seen, for comparison, in~\cite{MW}.

\begin{figure}[htbp] 
   \centering
   \includegraphics[width=.45\textwidth]{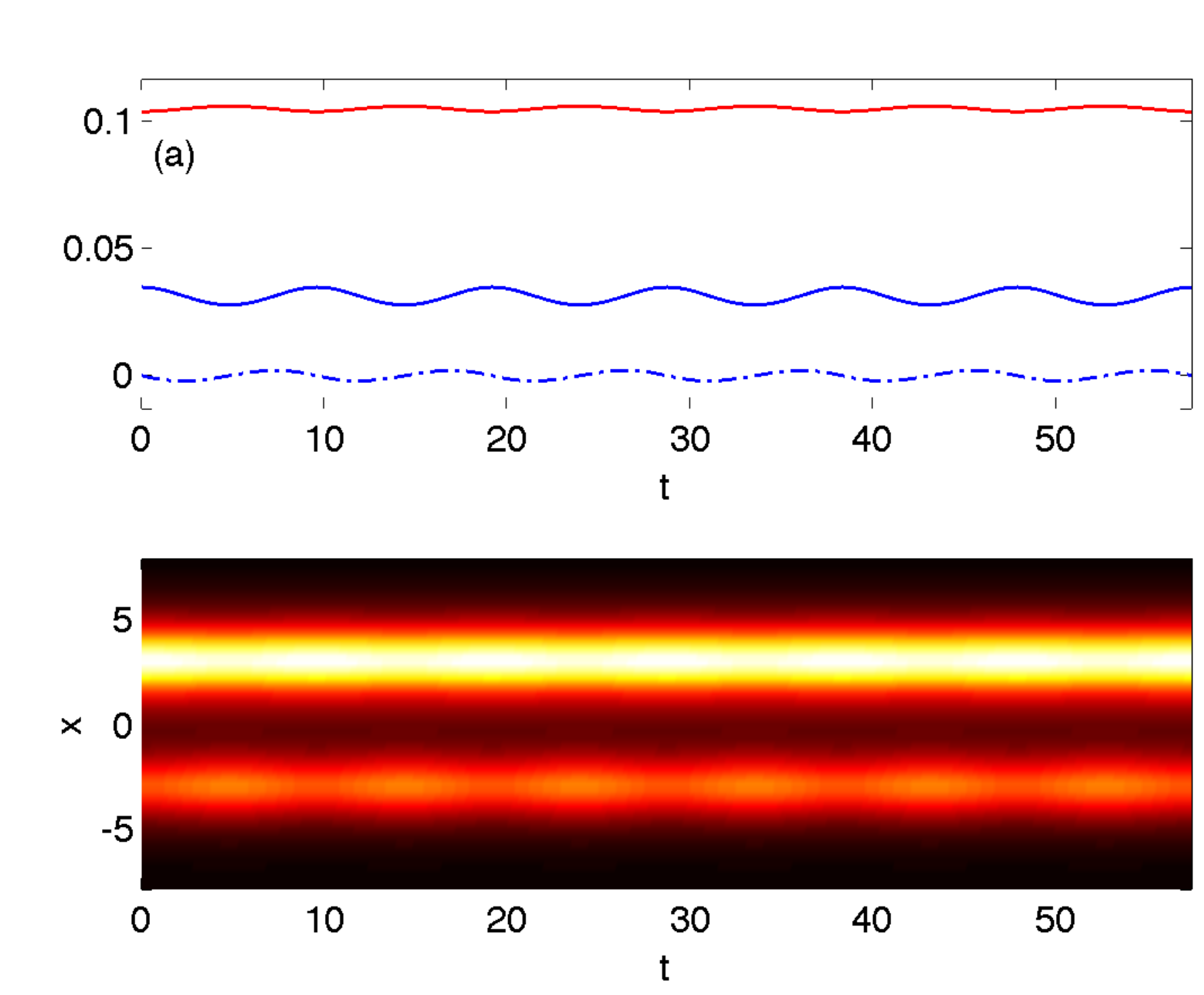}
   \includegraphics[width=.45\textwidth]{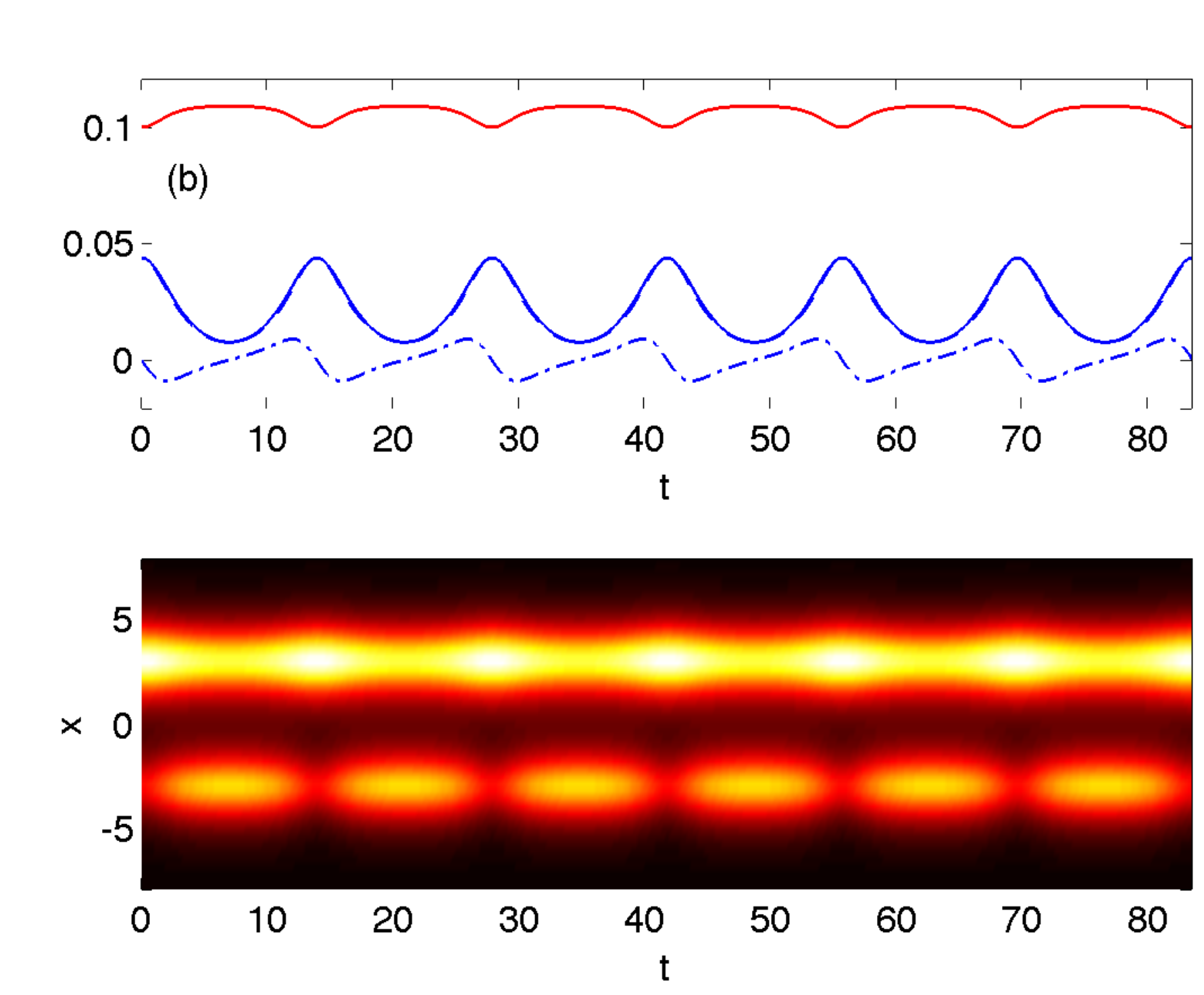}
   \includegraphics[width=.45\textwidth]{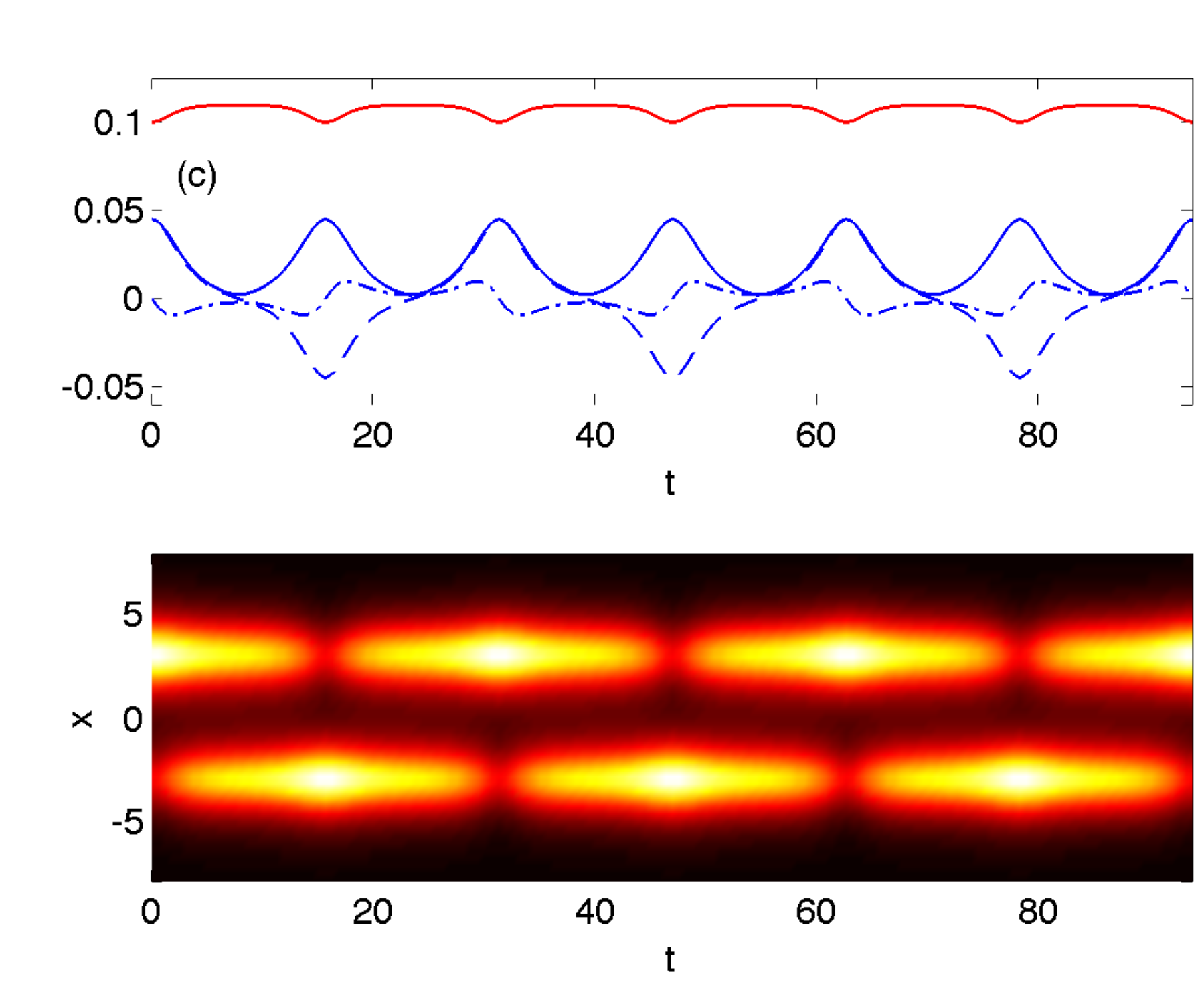}
   \includegraphics[width=.45\textwidth]{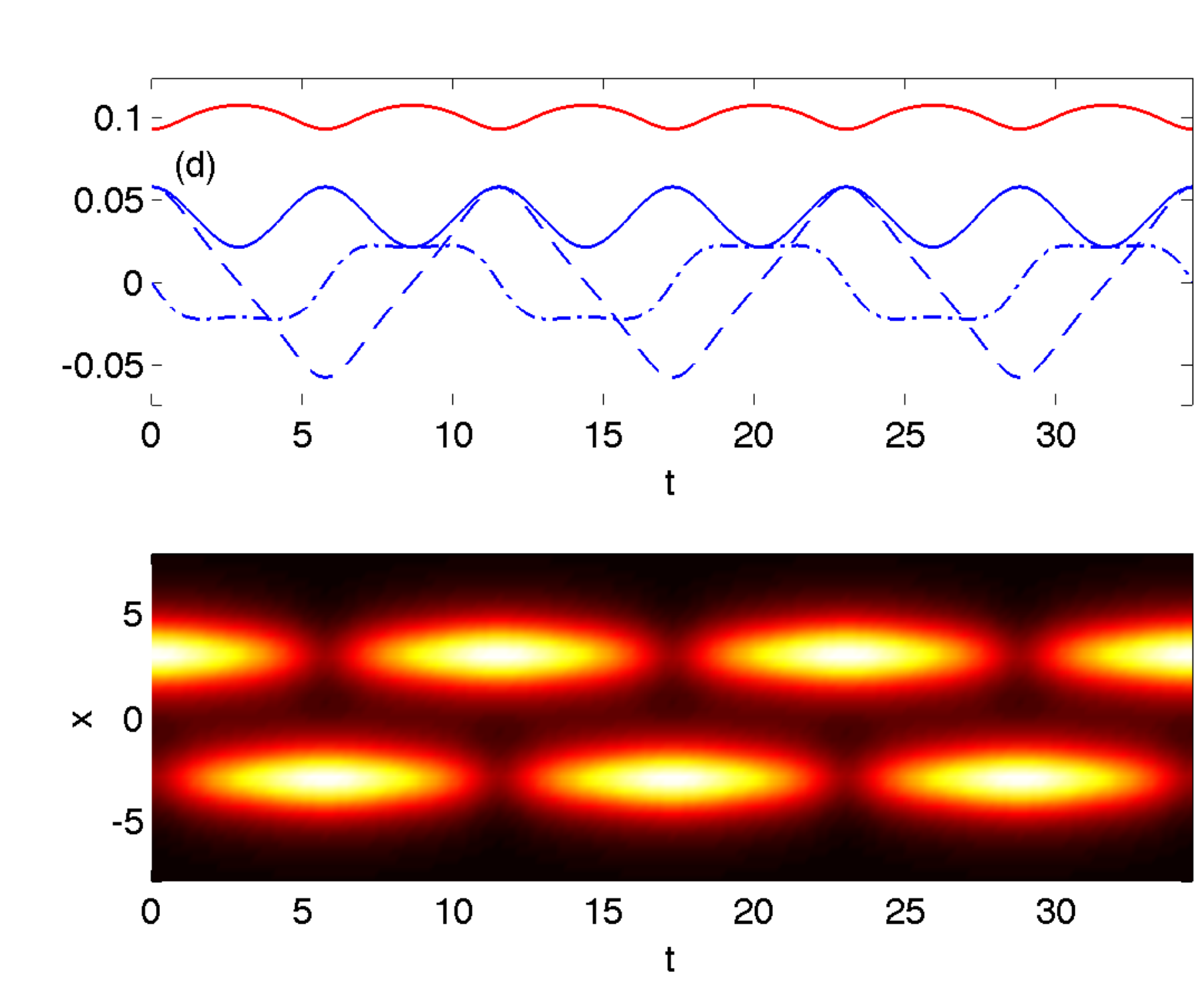}
   \caption{As in figure~\ref{reconstructed_sub} but for $N>\NcrFD$. \textbf{(a)} Near the fixed point $\alpha = \sqrt{-\sigma/2}$. \textbf{(b)} Just inside the separatrix. \textbf{(c)} Just outside the separatrix. \textbf{(d)} Far outside the separatrix. Note the difference in time scales between the various figures.}
\label{reconstructed_super}
\end{figure}

We now give a non-technical statement of the main result generalizing Theorem 5.1 of~\cite{MW}, which applies only to periodic orbits lying sufficiently close to stable equilibria of~\eqref{eqn:dy}.  A precise statement appears in Sec.~\ref{sec:maintheorem}, Theorem~\ref{thm:main-eq}.

\noindent {\bf Main Result:}
Let $0<|N-\NcrFD|$ be sufficiently small. For periodic solutions to the ODE system~\eqref{eqn:alphabeta} of sufficiently small amplitude for $N<\NcrFD$ or $N > \NcrFD$, as long as the periodic solutions are sufficiently bounded away from the separatrix  in the case $N>\NcrFD$, 
there are corresponding solutions to the NLS equation~\eqref{eqn:nlsdwp-gK} which shadow these orbits of the finite dimensional reduction on very long, but finite, time scales.

The remainder of this paper is organized as follows:
Section~\ref{sec:maintheorem} contains a precise formulation of the main result, Theorem~\ref{thm:main-eq}. Section~\ref{s:exact} presents exact formulae for the periodic solutions  of~\eqref{eqn:alphabeta}, from which we derive period and amplitude estimates in Section~\ref{sec:rescaled}.  These bounds are used to show that the periodic orbits or the finite dimensional dynamical system satisfy the assumptions of Theorem~\ref{thm:main-eq}. The section concludes with a discussion of how to apply these estimates together with the analysis of~\cite{MW} to prove the main theorem. Section~\ref{sec:discussion} contains concluding remarks and a discussion of future directions.
The proof in this paper takes advantage of the exact solutions, which are convenient, but are not an essential feature for such a shadowing result to hold. Appendix~\ref{a:lie} details the use of Lie transforms to construct a normal form for system~\eqref{eqn:alphabeta}.  This is a possible first step in completing the proof in the absence of exact solutions.

\section{The main theorem}
\label{sec:maintheorem}

We begin by describing the regions of parameter space and phase space in which the results hold. Our results apply to the regime
 \begin{equation*}
 \abs{N-\NcrFD} \ll \NcrFD \ll1.
 \end{equation*}
Recall from~\cite{MW} that when $\ell \gg 1$, we have
\begin{equation*}
\NcrFD = \bigO{\Omega_1 (\ell) - \Omega_0 (\ell) } 
= \bigO{e^{-\kappa \ell}}
\; \text{for}\; \kappa > 0.
\end{equation*} 
Hence, we will define an asymptotic parameter $\sigma$ such that
\begin{equation*}
\NcrFD = O(\dd^\gamma),  \  \dd \ll 1  \  (\ell \gg 1)
\end{equation*}
with $7/9<\gamma<1$ as specified in Theorem~\ref{thm:main-eq} and take initial $N$ such that
\begin{equation*}
 \abs{N-\NcrFD} \leq \sigma.
\end{equation*}  
The $7/9$ here is not sharp, but is a remnant of the asymptotic techniques in~\cite{MW}.

The main result of this
paper is the following theorem.  Note that the key difference between Theorem 5.1 of~\cite{MW} and our main result, Theorem 2.1 (below), is that we now omit Assumption 3 of~\cite{MW}. We will use the statement and part of the proof of~\cite{MW}, Theorem $5.1$ in Section~\ref{sec:recap}.

\begin{thm}
\label{thm:main-eq}
There exists $\dd_0 > 0$ sufficiently small such that for all $0<\dd<\dd_0$ and $\frac{7}{9} < \gamma < 1$. Denote by
$$
{\tilde{\chi}}_{*}(t) = \left(\tilde{A}(t), \tilde{\alpha}(t) , \tilde{\beta}(t) \right),
$$
a periodic solution of~\eqref{eqn:dy}, satisfying the following two assumptions:
\renewcommand{\labelenumi}{(H\arabic{enumi})}
\renewcommand{\theenumi}{(H\arabic{enumi})}
\begin{enumerate}
\item\label{hypothesis1} Its period $\Tper(\dd) $ is not too large, as a function of $\sigma$, specifically
\begin{equation}
\Tper(\dd) \lesssim {\left(\abs{\NcrFD - N} \NcrFD\right)}^{-\frac{1}{2}}
 = \dd^{-\frac{1+\gamma}{2}} ,
\label{Tdelta}
 \end{equation}
 \item The  fundamental matrix solution $M(t)$ of the  dynamics linearized about $\chi_*(t)$ satisfies the norm bound:
\begin{equation}
0 < s,t < \Tper(\dd)\; \text{implies} \, \norm{M(t) M^{-1} (s)} \leq C\ {\left(\frac{N+\NcrFD}{|N-\NcrFD|}\right)}^{\frac{1}{2}}\ =\ C\ \dd^{\frac{\gamma-1}{2}}.
\label{Mbound}
\end{equation}
\end{enumerate}
Fix $\epsilon >0$ sufficiently small.  Then, there exist $\delta_0$, $\delta_1 > 0$ depending upon $\epsilon$, $\gamma$ such that the following holds.  Consider initial data of the form
\begin{equation*}
u_0(x) = e^{i \theta (0) } 
\left(\  \tilde{A} (0) \psi_0(x)\ +\  [\tilde{\alpha} (0) + i \tilde{\beta}(0)] \psi_1(x)\  \right) 
\end{equation*}
with $\theta(0)\in\RR$  chosen arbitrarily.  Then, there exists solution $u(x,t)$ of~\eqref{eqn:nlsdwp-gK} shadowing the periodic orbit of the~\eqref{eqn:dy}, 
 of the form
\begin{equation}
\label{full_u}
\begin{split}
u(x,t) & = e^{i \theta (t)} \left( (\tilde{A} (t) + \eta_A(t)) \psi_0(x)
  + [(\tilde{\alpha} (t) + \eta_\alpha(t)) + i ( \tilde{\beta} (t) +
  \eta_\beta(t))] \psi_1(x)  \right. \\
& \hspace{3cm} \left. +  \tilde{R}(x,t) + w(x,t)   \right).
\end{split}
\end{equation}
Here, $\theta(t) \in C^1([0,T_*(\dd)])$, and the remainder terms $\eta_A, \eta_\alpha, \eta_\beta, w$, and $\tilde{R}$ have the following properties:

\renewcommand{\labelenumi}{(B\arabic{enumi})}
\renewcommand{\theenumi}{(B\arabic{enumi})}
\begin{enumerate}
\item 
The function $\tilde{R}$ satisfies
\begin{equation*}
{\norm{\tilde{R}}}_{L^\infty_{t,x} } \lesssim \sigma^{1 + \delta_0}
\end{equation*}
and is the solution of
 \begin{equation}
\label{eqn:tildeR}
i \tilde{R}_t = (H - \Omega_0) \tilde{R} + (\tilde{A}^2 + 3 \tilde{\alpha}^2 + \tilde{\beta}^2) \tilde{R} + P_c F_b({\chi}_{*}) ,
\end{equation} 
where $F_b$ is displayed in Equation $(2.5)$ and Appendix $A$ of~\cite{MW}. 

\item 
We have $\eta (t)\equiv\left(\ \eta_A(t),\eta_\alpha(t),\eta_\beta(t)\right)\in C^1([0,T_*(\dd)])$  and  $w(t,x) \in L^\infty_t H^1_x \cap L^4_t L^\infty_x$ with the bounds
\begin{equation*}
\| {\eta} \|_{L^\infty_t[0,T_*(\dd)]} + \| w \|_{L^\infty_t([0,T_*(\dd)]; H^1_x)} + \| w \|_{L^4_t([0,T_*(\dd)]; L^\infty_x)} \lesssim \sigma^{\frac{1}{2}+\delta_1},
\end{equation*}
for all $t \in I =[0,T_*(\dd)]=[0,\Tper(\dd)\ \dd^{-\epsilon}]$.
\end{enumerate}
\end{thm}


\begin{rem}
The full system we solve in Theorem~\ref{thm:main-eq} can be found in equations $(2.5)-(2.6)$ and $(2.11)-(2.17)$, with all error terms written out in Appendix A of~\cite{MW}.  Up to a phase shift, the phase term $\theta(t)$ in equation~\eqref{full_u} evolves under an equation similar to~\eqref{varthetadot} for $\vartheta(t)$ in equation~\eqref{eqn:c0c1}, though it differs by a small interaction term stemming from error terms in the full PDE expansion.  This leads to $\bigO{1}$ differences on the time scales considered. This is discussed in greater detail in appendix A and Section $3.1$ of~\cite{MW} and revisited in Section~\ref{sec:recap} below.  The $\delta_0$ and $\delta_1$ are sufficiently small, but non-zero constants chosen such that the bootstrapping arguments of~\cite{MW} hold.  In particular, they arise from the fact that the error terms must remain of lower order than the dominant dynamics over the time scale we study.
\end{rem}

The proof that there exist periodic orbits of the finite dimensional model that are shadowed by solutions to NLS is then provided by the following Lemma:
\begin{lem}
\label{lem:ampbds}
There exists $\dd_0>0$ and $\delta > 0$ (chosen as in~\cite{MW}), such that for $0<\dd<\dd_0$, the system~\eqref{eqn:alphabeta} has periodic solutions $(\tilde{\alpha}, \tilde{\beta}) (t)$ which through equation~\eqref{N} determine a periodic function $A(t)$ such that:
\begin{align*}
|\tilde{A}(t)|^2 &=O(\dd^\gamma),\quad |\tilde{\alpha}(t)|^2 +
|\tilde{\beta}(t)|^2\ =O(\dd),\qquad N>\NcrFD = \dd^\gamma, \\
|\tilde{A}(t)|^2 &=O(\dd^\gamma),\quad |\tilde{\alpha}(t)|^2 +
|\tilde{\beta}(t)|^2\ =O(\dd^{1+\delta}),\qquad N<\NcrFD = \sigma^\gamma.
\end{align*}

The periods of these orbits satisfy the  bound~\eqref{Tdelta} required in Theorem~\ref{thm:main-eq}.  Furthermore, the fundamental solution matrix $M(t)$ for the linearized system (see systems $(3.37)$, $(3.39)$ of~\cite{MW}) about these periodic orbits satisfies the bound~\eqref{Mbound} from Theorem~\ref{thm:main-eq}.
\end{lem}

We next embark on the proof of this theorem, beginning with constructing periodic orbits and bounding their periods.

\section{Explicit Construction of Periodic Orbits}
\label{s:exact}

To solve the ODE system~\eqref{eqn:alphabeta}, we first obtain a simpler form of the equations by the following rescaling:
\begin{equation}
\label{eqn:rescaling}
\alpha = \sqrt{\frac{\W_{10}}2 } q (\tau), \, \beta =  \sqrt{\frac{\W_{10}}2 } p (\tau), \,
 \tau =  \W_{10} t.
\end{equation}
In terms of the new variables $p,q$, the Hamiltonian (up to the addition of a constant) is given by
\begin{align}
H_{\rm DW} (q,p) &= \frac{1}{2}(1+ q^2) p^2 + \frac{1}{2}{\left(q^2 - \frac{\zeta}{2} \right)}^2 
\equiv T(q,p) + V(q), \notag\\  
\zeta &= (2N - \W_{10})/\W_{10} = \bigO{\dd^{1 - \gamma}}.
\label{eqn:zetadef}
\end{align}
The potential energy $V(q)$ has double-well structure in the supercritical case $\zeta>0$ ($N > \NcrFD$).  Note that the kinetic energy depends both on the position and momentum.
When $\zeta>0$, the system has three fixed points, a saddle at $q=0$, and centers at $q= \pm \sqrt{\tfrac{\zeta}{2}}$.

From $\dot{q} = \frac{\partial H}{\partial p} = (1+ q^2)p$, we obtain $p = {\dot q}/{(1+ q^2)}$.
By conservation of Hamiltonian, a solution with initial condition $(q(0),p(0)) = (q_0,0)$ satisfies
\begin{equation*}
 H(\tau) = H(0) = \frac{1}{2}{\left(q_0^2 - \frac{\zeta}{2} \right)}^2 
= \frac{1}{2}(1+ q^2) p^2 + \frac{1}{2}{\left(q^2 - \frac{\zeta}{2} \right)}^2 \\
=\frac{1}{2}\frac{\dot{q}^2}{1+  q^2} + \frac{1}{2}{\left(q^2 - \frac{\zeta}{2}\right)}^2.
\end{equation*}
Solving for $\dot{q}^2$ yields
\begin{equation}
\label{eqn:qdot1}
\dot{q}^2   = 
\left(1+ q^2\right)\left({\left(q_0^2 - \frac{\zeta}{2}\right)}^2 - {\left(q^2-\frac{\zeta}{2}\right)}^2 \right) \\
 = \left(1+ q^2\right) \left(q_0^2 - q^2 \right) \left(q^2 + q_0^2 -  \zeta \right).
\end{equation}
We now proceed with the exact integration of~\eqref{eqn:qdot1}.

\subsubsection*{Josephson Tunneling and Nonlinear Beating Solutions}

Let us first analyze orbits exterior to the separatrix for $N > \NcrFD$.  By~\eqref{eqn:zetadef}, this implies that $q_0-\sqrt{\zeta} > 0$.  
Define $a$ by $$a^2 \equiv q_0^2-\zeta.$$ 
Equation~\eqref{eqn:qdot1} may be rewritten as
\begin{equation*}
\dot{q}^2 =
\left(1+  q^2\right)\left(q_0^2-q^2\right)\left(a^2 + q^2\right).
\end{equation*}
Separation of variables and integration yields
\begin{equation*}
\tau = - \int_{q_0}^{q} 
\frac{d{\tilde{q}}}
       {\sqrt{\left(1+  {\tilde{q}}^2\right)\left(a^2 + {\tilde{q}}^2\right)\left(q_0^2-{\tilde{q}}^2\right)}}
       =\frac{1}{\sqrt{(1-a^2)(1+q_0^2)} } 
\int_u^B \frac{dQ}{\sqrt{(A^2 + Q^2)(B^2-Q^2)}},
\end{equation*}
where we have defined
$$
{\tilde{q}} = \frac{Q}{\sqrt{1-Q^2}}, \,
A^2 = \frac{a^2}{1-a^2}, \, 
B^2 = \frac{q_0^2}{1+q_0^2}, 
\text{and } 
u^2 = \frac{q^2}{1+q^2}.
$$
A further simplification is reached via the trigonometric substitution $Q=B \sin{\theta}$ (see also~\cite[integral 3.152.4]{Gradshteyn:2007})
$$
\int_u^B \frac{dQ}{\sqrt{(A^2 + Q^2)(B^2-Q^2)}} = \frac{1}{\sqrt{A^2+B^2}} F(\delta,k),
$$
where $F(\delta, k)$ is the \emph{incomplete elliptic integral of the first kind} given by
\begin{equation*}
F(\mu,k) = \int_0^\mu \frac{dt}{\sqrt{1-k^2 \sin^2 t}}, \ 0 \le k <1,
\end{equation*} 
and
\begin{equation}
 k^2 = \frac{B^2}{A^2 + B^2} = \frac{\left(1+ \zeta-q_0^2\right)q_0^2}{2 q_0^2-\zeta}\, \text{and }\, \delta = \cos^{-1}\frac{u}{B} .
\label{k-def}
\end{equation}
Some calculation yields the formula 
$$
\omega \tau = F(\delta, k), \text{where } \omega = \sqrt{q_0^2- \frac{\zeta}2}.
$$

To obtain $q$ as a function of $\tau$ requires an identity that inverts the equation $\theta = F(\delta, k)$.
This inverse defines the Jacobi elliptic function:
$$
\sn(\theta,k) = \sin{\delta}.
$$
An introduction to these quantities in the language of modern dynamical systems is~\cite{Meyer:2001}, while a comprehensive handbook is by Byrd and Friedman~\cite{Byrd:1954}.
The functions $\cn$ and $\dn$ are defined in a similar manner.
From here, there remains only the algebra to invert all the above changes of variables, which makes use of the identity $\sn^2(\delta,k) + \cn^2(\delta,k) = 1$, and eventually arrives at
$$
q(\tau) = \frac{q_0 \cn{(\omega \tau,k)}}{\sqrt{1+q_0^2 \sn^2{(\omega \tau,k)}}}.
$$
The orbit $q(t)$ has period given in terms of $K(k)$, the \emph{complete elliptic integral of the first kind}
\begin{equation}
T = \frac{4 F\left(\frac{\pi}{2},k\right)}{\omega}= \frac{4 K(k)}{\omega} = \frac{4}{\omega} K\left(q_0 \sqrt{\frac{1+ \zeta-q_0^2}{2q_0^2-\zeta}}\right).
\label{eqn:Tperiod}
\end{equation}
This is a decreasing function of $q_0$ for $q_0>\sqrt{\zeta}$ and fixed $\zeta$.  
For the case $N < \NcrFD$, we note that $\zeta <0$, hence the analysis follows in an identical fashion.

\subsubsection*{Self-trapping solutions}
The periodic orbits encircling only one fixed point in the supercritical case $\zeta>0$ may be found in the same manner. These orbits are of the form
\begin{gather*}
q(t) =\pm  \frac{q_0\dn{(\omega t, k)}}{\sqrt{1+q_0^2 k^2 \sn{(\omega t, k)}^2}},\\ 
\text{with }\sqrt{\zeta/2}<q_0<\sqrt{\zeta},  \;
\omega^2  = \frac{\left(1+\zeta-q_0^2\right)q_0^2}{2}, \, \text{and } \;
k^2 = \frac{2q_0^2-\zeta}{q_0^2 \left(1+\zeta-q_0^2\right)}.
\end{gather*}
Then, once again we have 
\begin{equation*}
T = \frac{2 F\left(\frac{\pi}{2},k\right)}{\omega}= \frac{2 K(k)}{\omega} = \frac{2}{\omega} K\left(q_0 \sqrt{\frac{2q_0^2-\zeta}{q_0^2 \left(1+\zeta-q_0^2\right)}}\right).
\end{equation*}

\section{Bounds on the monodromy matrix and the period for Josephson Tunneling orbits}
\label{sec:rescaled}

\subsection{Verification of Hypothesis~\ref{hypothesis1} from Theorem~\ref{thm:main-eq} and amplitude bounds from Lemma~\ref{lem:ampbds}}
\label{sec:peramp}

We now prove the period estimate~\eqref{Tdelta} from the expression~\eqref{eqn:Tperiod}.  Each orbit of~\eqref{eqn:dy} has an initial condition of the form 
$$
\tilde{\chi}(0)=(\alpha_0, 0, A_0),
$$
where $\alpha_0^2 + A_0^2 > \NcrFD = \dd^\gamma$.
The exact solution $\alpha (t)$ and $\beta(t)$ to system~\eqref{eqn:dy} is given by equations~\eqref{eqn:rescaling}.  Then, $A(t)$ is computed from~\eqref{N}.  
The values of $q_0$ and $\zeta$ in terms of $\dd$ are read off these expressions.

Using the asymptotic parameter $\sigma$ as set in~\eqref{sigma}, we have that the parameters $\Omega_{10} = \bigO{\sigma^\gamma}$ and $\zeta = \bigO{\sigma^{1-\gamma}}$.
With this choice of $\Omega_{10}(\sigma)$ and $\zeta(\sigma)$, Hypothesis~\ref{hypothesis1} of Theorem~\ref{thm:main-eq} holds for initial values, $q_0$, for which
\begin{equation}
\label{eqn:orbitTdef}
T = \frac{4 K\left(k(\zeta,q_0)\right)}{\omega \Omega_{10}} = \bigO{\sigma^{- \frac{1+\gamma}{2}} },
\end{equation}
where $k(\zeta,q_0)$ is given by~\eqref{k-def}.

To satisfy~\eqref{eqn:orbitTdef}, we take 
\begin{equation}
C \geq q_0 \geq  \sqrt{\frac{3\zeta}{2}} \sim \sqrt{\frac{3}{2}}
\bigO{\sigma^{\frac{1-\gamma}{2}}}, 
\label{eqn:q0bds}
\end{equation}
for some $C(\sigma)$ bounded from above but that we do not attempt to optimize here.  Then, $\omega=\sqrt{q_0^2-\zeta/2}\approx \sigma^{\frac{1-\gamma}2}$.
Therefore,  $\omega\Omega_{10}\approx \sigma^{\frac{1+\gamma}2}$. 
To prove the bound  
\eqref{eqn:orbitTdef} it suffices to show $K\left(k(\zeta,q_0)\right)=\bigO{1}$.

The complete elliptic integral has the asymptotic behavior $K(k) \to \infty$ logarithmically as $k^2 \to 1$, where $k^2-1$ in effect encodes the distance of the exact solution from the separatrix. Hence, we must choose $q_0$ sufficiently far from the separatrix such that $K(k)$  is uniformly bounded with a computable constant dependent upon $k_0$ when $|k| < k_0 < 1$.  Indeed, we observe via a Taylor expansion of $k$ in the parameter $\zeta/2q_0^2$ in equation~\eqref{k-def} that
$$
k \leq \frac{7}{8} \, \text{for} \ q_0 \geq \sqrt{\frac{3}{2}} \sigma^{\frac{1-\gamma}{2}}.
$$ 
More specifically, given the scaling $\sqrt{\zeta} \ll q_0 \ll \sigma$, equation~\eqref{k-def} yields the approximation
\begin{equation*}
k \approx \frac{1}{2} +  \frac{\zeta}{4q_0^2} ,
\end{equation*}
which for $q_0$ as in~\eqref{eqn:q0bds} and  $\sigma$ sufficiently small implies 
\begin{equation*}
k \leq \frac{1}{2} +  \sum_{j=1}^\infty \frac{\zeta}{4q_0^2}  + \mathcal{O} (\sigma^{\frac{1-\gamma}{2}}) \leq \frac{3}{4} + \mathcal{O} (\sigma^{\frac{1-\gamma}{2}})  < \frac{7}{8}.
\end{equation*}

It follows that
$$T < \frac{4 K(7/8)}{\omega \Omega_{10}} = \bigO {\sigma^{- \frac{1+\gamma}{2}}}  \sim T_{\rm shadow}, \ \text{using} \ \omega = \sqrt{q_0^2 - \zeta/2}, \ \Omega_{10} = 2 \NcrFD.$$
At this stage we note that as $K$ is a decreasing function as $k \to 0$, hence orbits orbits sufficiently close to the separatrix violate the period bound~\eqref{Tdelta}.  Also, the choice of $7/8$ is not  sharp, but suffices to bound the term  $3(1+\sigma)/4$ arising as the leading order of the Taylor expansion of $k$ using~\eqref{eqn:q0bds}.

To verify the amplitude bounds from Lemma~\ref{lem:ampbds} for $q_0$ satisfying estimate~\eqref{eqn:q0bds}, we directly evaluate~\eqref{eqn:rescaling} and observe
\begin{equation}
\label{eqn:ampbds1}
|\alpha|, |\beta| \lesssim \sigma^{\frac{1}{2}}
\end{equation}
which, using the  conservation equation~\eqref{N}, where $N - \alpha^2 - \beta^2 \sim \sigma^\gamma$ gives 
\begin{equation}
\label{eqn:ampbds2}
|A| \lesssim \sigma^{\frac{\gamma}2}.
\end{equation}

\subsection{Verification of hypothesis $H2$ from Theorem~\ref{thm:main-eq}}
\label{sec:fundsol}

Equations (3.36)--(3.39) of~\cite{MW} discuss the linearization of system~\eqref{syst-rot} about an arbitrary periodic solution, 
$$\chi_*=(\alpha_*,\beta_*,A_*,\theta_*)$$ 
with period $T_*$.   Separating the linearized evolution component into the coupled $(\alpha,\beta,A)$ system~\eqref{eqn:dy} and the independent $\theta$ evolution equation yields
\begin{subequations}
\label{eqn:sys-lin}
\begin{align}
\label{eqn:sys-lin-alpha}
\frac{d}{dt} \begin{bmatrix}
\delta \alpha \\
\delta \beta \\
\delta A 
\end{bmatrix}
&=  
\begin{bmatrix}
4 \alpha_* \beta_* & 2(\NcrFD + \alpha_*^2) & 0  \\
-(\Omega_{10} +6\alpha_*^2-2A_*^2)  & 0 & -4 \alpha_*A_*  \\
- 2 A_* \beta_* & -2 \alpha_* A_* & -2 \alpha_* \beta_* \\
\end{bmatrix}
\begin{bmatrix}
\delta \alpha \\
\delta \beta \\
\delta A
\end{bmatrix}
\equiv  B (t) \begin{bmatrix}
\delta \alpha \\
\delta \beta \\
\delta A
\end{bmatrix};  
\\
\label{eqn:sys-lin-theta}
\frac{d}{dt} (\delta \theta) & = 
\begin{bmatrix}
6\alpha_* & 2\beta_* & 2A_* 
\end{bmatrix} 
\cdot 
\begin{bmatrix}
\delta \alpha \\
\delta \beta \\
\delta A 
\end{bmatrix}.
\end{align}
\end{subequations}
Floquet's theorem says that there exists a matrix-valued $T_*$-periodic function $P(t)$ and a constant matrix $B_*$, such that the fundamental solution matrix of subsystem~\eqref{eqn:sys-lin-alpha} with initial condition $M(t)=0$ is
\begin{equation*}
M(t) = P(t) e^{B_* t} \; \text{with } \ P(0) = P(T_*) = I.
\end{equation*} 

In~\cite{MW}, Section $3.3$, the authors construced $M_* = e^{B_* t}$ at a stable equilibrium point, then restricted their analysis to nearby periodic orbits, for which they could prove bounds on $M(t)$ perturbatively.  Here, we explicitly construct the matrix $M$ at any periodic orbit using the exact solution derived in Section~\ref{s:exact}.   The linearly independent column vectors generating the matrix $M$ are found by differentiating equation~\eqref{eqn:dy} with respect to the canonical system parameters $\zeta$, $q_0$ and $t$.  These represent translation in the mutually orthogonal directions through energy space (translation perpendicular to the phase plane through energy space), within the energy plane onto a nearby orbit (translation in the normal direction to an orbit) and along a heteroclinic orbit respectively (translation in the tangential direction to an orbit).  

To proceed, we evaluate the periodic solution ${\chi}(t)$ and then its derivatives $\partial_t \chi$, $\partial_{q_0} \chi$, and $\partial_\zeta\chi$ using the exact solution.  Each of these vectors  solves equation~\eqref{eqn:sys-lin-theta}. A matrix whose columns  are given by these three vectors, or, in fact by any three independent linear combinations of these solutions, defines a fundamental solution matrix. Of these three vectors, the first is $T$-periodic in time, while the second and third exhibit linearly-growing oscillations.
 
We first compute
$$
\partial_{t}{\chi}
= \begin{bmatrix}
\sigma^{\frac{1}{2} } \dot{q} (\sigma^{\gamma} t) \\
\sigma^{\frac{1}{2} } \dot{p} (\sigma^{\gamma} t) \\
- (\alpha \dot \alpha +  \beta \dot \beta) {\left(N - \alpha^2 - \beta^2\right)}^{-\frac{1}{2}}
\end{bmatrix}.
$$
At $t=0$, only the second component is nonzero, so we can define the renormalized vector
$$
{\chi}_t^{\ren}(0) = 
 \begin{bmatrix}
0 \\
1 \\
0
\end{bmatrix}.
$$

The vectors $\partial_{q_0}{\chi} $ and $\partial_{\zeta}{\chi}$ can be computed similarly.  As the expressions for them in general is rather long, we display here only their values at $t=0$:
$$
\left.\partial_{q_0 }{\chi} \right|_{t=0} = 
 \begin{bmatrix}
\frac{1}{\sqrt{2}} \W_{10}^{\frac{1}{2}} \\
0  \\
 - q_0 \W_{10} A^{-1} (0)
\end{bmatrix} 
\text{and  }
\left.
\partial_{\zeta}{\chi} \right|_{t=0} = 
\begin{bmatrix}
0 \\
0 \\
\frac{1}{2} (1-q_0^2)  A^{-1} (0)
\end{bmatrix}.
$$
Using from Section~\ref{sec:peramp} that $A(0) = \bigO {\sigma^{\frac{\gamma}{2}} }$, these may be renormalized to
$$
\chi_{q_0}^{\ren} (0) =
 \begin{bmatrix}
1 \\
0  \\
- C \ \dd^{\frac{1-\gamma}{2}}
\end{bmatrix} 
\text{and  }
\chi_{\zeta}^{\ren} (0) =
\begin{bmatrix}
0  \\
0  \\
1
\end{bmatrix},
$$
where $C$ is a positive $\bigO{1}$ constant.  We can then define the fundamental solution matrix
\begin{equation}
\label{eqn:tilMtdef}
\tilde{M}(t) = 
\left[  
\chi_{q_0}^{\ren} (t) \ 
\chi_{t}^{\ren} (t) \
\chi_{\zeta}^{\ren} (t)
\right],
\end{equation}
and the matrix 
\begin{equation}
\label{eqn:Mtdef}
M (t) = \tilde{M}(t) {\tilde{M}}^{-1} (0).
\end{equation}
Note, since $M(0)=I$, $M(T)$ is the  monodromy matrix for the Floquet system~\eqref{eqn:sys-lin-alpha}.

To prove~\eqref{Mbound} for the fundamental solution operator $M(t)$, we prove first that the growth over one full period is bounded.  Then, we prove that the variation within a single period is bounded.

\begin{lem}
\label{lem:eigen}
The eigenvalues of the monodromy matrix $M(T)$ are $\lambda_1 = \lambda_2 = \lambda_3=1$. As a consequence the eigenvalues of $B_*$ are all zero. In addition, $M(T)$ has at least two linearly independent eigenvectors, so that $\norm{e^{B_*T}}< 1+ c T$.
\end{lem}

\begin{proof}
This lemma is essentially Proposition $3.3$ in~\cite{MW}, but we recall the proof here for completeness.  It is clear that $\dot{\chi}=(\dot \alpha(t),\dot \beta(t), \dot A(t))$ is a solution to~\eqref{eqn:sys-lin-alpha}, giving at least one Floquet multiplier $\lambda_1 = 1$. Similarly, differentiation with respect to the period gives a similar result, implying $\lambda_1 = \lambda_2 = 1$.

The exact solutions  $\alpha$ and $\beta$,
satisfy the symmetries $\alpha (t) = \alpha (T-t)$ and $\beta (t) = -\beta (T-t)$.  Hence
\begin{equation}
\int_0^T \alpha (s) \beta(s) ds = 0.
\label{eqn:int_alpha_beta}
\end{equation}

In Floquet theory, Liouville's formula provides the determinant of the monodromy matrix $M_L$ for the system
$$
\dot{X} = L(t) X
$$
for $L(t+T) = L(t)$. It states
$$
\det M_L (T) = e^{\int_0^T {\rm Tr} (L) (s) ds}.
$$
Hence, given ${\rm tr} (B(t)) =2 \alpha(t) \beta(t)$, equation~\eqref{eqn:int_alpha_beta} implies
$\lambda_1 \cdot \lambda_2 \cdot \lambda_3 = 1$ and, as a result, $\lambda_3 = 1$.  

Note, using a similar approach to above, there exist at least two linearly independent vectors by using differentiation in $t$ and $T$.  

\end{proof}

The above lemma shows that the $n$-times iterated monodromy operator satisfies the bound  
$$
\norm{M(nT)} \le 1 + C (n  T)  = \bigO{n  \sigma^{- \frac{1+\gamma}{2}} }
$$
for some positive constant $C$.
In order to show that the solution to~\eqref{eqn:nlsdwp-gK} stays close to the finite dimensional orbit $\chi_*$ on the time scale of shadowing, we require that the norm of $M(t)$ for $0< t < T$ does not grow too large to prevent us from applying the bootstrapping arguments in the proof of the main theorem from~\cite{MW}.  Hence, we require slightly better control on the growth of $M$ than we have currently shown.  Indeed, we require the following lemma bounding the operator within a the evolution of a single period.

\begin{lem}
\label{lem:fundsolbd}
Let $M(t)$ be the fundamental solution matrix defined by~\eqref{eqn:Mtdef},
then for $0<s,t<T \sim \sigma^{-\frac{1+\gamma}{2}}$, $\norm{M(t)} = \bigO{\sigma^{\frac{\gamma-1}{2} }}$.
\end{lem}

\begin{proof}
Note, it suffices to prove this theorem for $\tilde{M}$ in~\eqref{eqn:tilMtdef}.  To prove this, we may simply calculate the entries in $\tilde{M}(t)$ and find the maximum magnitude of each as a function of time over all values of $0 \le t < T$. To begin, we recall the functional form the exact solution for $\alpha (t)$:
\begin{equation*}
\alpha (t) = \sqrt{\frac{\Omega_{10}}{2}} q_0 \frac{\cn (\omega \Omega_{10} t, k ) }{\sqrt{1 + q_0^2 \sn^2 (\omega  \Omega_{10} t, k ) } }.
\end{equation*}
There are similar expressions for $\beta(t)$ and $A(t)$.  We will make use of the fact that $\cn(z,k)$, $\sn(z,k)$, and their $z$-derivatives are bounded functions for all $k$ and $z$, and that derivatives with respect to $k$ exhibit linearly-growing oscillations for fixed $k$ and increasing $z$.  These bounds take advantage of the extensive literature on Jacobi elliptic functions~\cite{Byrd:1954,NIST:DLMF}.

Since we have chosen $q_0$ and $\omega$ such that $k$, $K$ and $K'=\tfrac{dK}{dk}$ are uniformly bounded, we can find similar bounds for $\beta$ and $A$ but must be mindful of the dependence on the asymptotic parameter $\sigma$ in constructing the fundamental solution matrix.  
We observe that the fundamental solution matrix can be represented in terms of the column vectors defined in~\eqref{eqn:Mtdef}.
Note, under such a choice 
\begin{align*}
\tilde{M} (0 ) &\equiv I_3 - C \sigma^{\frac{1}{2}-\frac{\gamma}{2}} 
\begin{bmatrix}
0 & 0 & 0 \\
0 & 0 & 0 \\
1  & 0 & 0
\end{bmatrix}\\
\intertext{and}
{\tilde{M}}^{-1} (0 ) &\equiv I_3  + C \sigma^{\frac{1}{2}-\frac{\gamma}{2}}   
\begin{bmatrix}
0 & 0 & 0 \\
0 & 0 & 0 \\
1  & 0 & 0
\end{bmatrix}.
\end{align*}
A calculation relying upon boundedness of $\cn$, $\sn$ and their derivatives in $k$ reveals that in addition to the leading-order terms at $t = 0$ discussed above, the implicit differentiation in the columns of $M(t)$ does not introduces terms that do not lead to linear growth in $t$ (i.e.\ purely periodic solutions) and terms that have linear growth as observed in the orbits in~\cite{MW}, Proposition $3.3$.  

The terms without linear growth have coefficients of the form:
\begin{equation}
\begin{split}
\label{eqn:percoeffs}
 \frac{\sqrt{\Omega_{10}} q_0}{\Tper},  \;   
 \frac{-\sqrt{\Omega_{10}} q_0}{k} \frac{\partial k }{\partial \zeta}, \; 
\frac{-\sqrt{\Omega_{10}} q_0 }{K (k)} \frac{\partial k }{\partial \zeta},  \; 
 \frac{-\sqrt{\Omega_{10}} q_0 }{k}\frac{\partial k }{\partial q_0} , \\ 
 \text{and }
  \frac{-\sqrt{\Omega_{10}} q_0}{K (k)} \frac{\partial k }{\partial q_0}, \; 
\sqrt{\Omega_{10}} q_0 \kappa' (k)  \frac{\partial k }{\partial \zeta}, \;  
\sqrt{\Omega_{10}} q_0 \kappa '(k) \frac{\partial k }{\partial q_0}, 
\end{split}
\end{equation}
 and have bounds maximum bounds given by
 \begin{equation}
\label{eqn:tdepbds}
 \frac{\partial k }{\partial q_0} = \bigO{\sigma^{\frac{\gamma-1}{2} } },  \frac{\partial k }{\partial \zeta} = \bigO{\sigma^{\gamma-1 } }.
 \end{equation}
The terms that do grow linearly in $t$ are of the form:
\begin{equation}
\label{eqn:tcoeffs}
\frac{\Omega_{10}^{3/2} q_0}{K(k)}  \frac{\partial \w}{\partial {q_0}}  t, \;
\frac{\Omega_{10}^{3/2} q_0}{K(k)}  \frac{\partial \w}{\partial \zeta}  t, \\ 
\frac{\Omega_{10}^{3/2} q_0  K'(k)}{K^2(k)}  \frac{\partial k}{\partial {q_0} }  \omega t , \;
\text{and }
\frac{\Omega_{10}^{3/2} q_0  K'(k)}{K^2(k)}  \frac{\partial k}{\partial \zeta }  \omega t.
\end{equation}

Noticing that
\[
\norm{M(t)} \leq \norm{\tilde{M}(t) }\norm{\tilde{M}^{-1}(0) },
 \]
on the time scale of interest and keeping track of the renormalization constants, we observe the fundamental solution bound is given by, for instance, the largest of the time-dependent terms.  Hence,
\begin{align*}
\norm{\tilde{M} (t) } & \lesssim  \Omega^{\frac{3}{2}}_{10} q_0 \omega t  \partial_{\zeta} k \left(  K^{-2} (k) K'(k) \right)  \\
& \hspace{.1cm} \lesssim \sigma^{\frac{3\gamma}{2}} \sigma^{\frac{1-\gamma}2 }  \sigma^{\frac{1-\gamma}{2}} \sigma^{- \frac{1+\gamma}{2}} \sigma^{-1 + \gamma}  \mathcal{O} (1)  \\
& \hspace{.2cm} \lesssim \bigO{\sigma^{\frac{\gamma-1}{2} } },
\end{align*}
using 
\begin{equation*}
\Omega_{10} \sim \sigma^{\gamma}, \ \omega \sim \sigma^{\frac{1-\gamma}{2}}, \ t < \Tper = \sigma^{- \frac{1+\gamma}{2}} 
\end{equation*}
in addition to~\eqref{eqn:tdepbds}.  This is the desired bound for $M(t)$ given $0\leq t \leq T$.  
\end{proof}

\subsection{Brief Recap of the Perturbative Methods from~\cite{MW}}
\label{sec:recap}

Lemma~\ref{lem:fundsolbd} implies that the evolution within one period is also bounded, hence the Duhamel evolution operator, $M(t)M^{-1}(s)$ for $0< t-s <T$, used in the perturbation theoretic arguments from~\cite{MW} has the same bound.  Coupling this implicit Duhamel evolution bound, the period bound from~\eqref{eqn:orbitTdef}, and the amplitude bounds from~\eqref{eqn:ampbds1} and~\eqref{eqn:ampbds2} shows that assumptions~\eqref{Tdelta},~\eqref{Mbound} from Theorem~\ref{thm:main-eq} and the amplitude bounds in Lemma~\ref{lem:ampbds} all hold for the exact solution.  Hence, the dynamical solutions satisfy the assumptions of Theorem $5.1$ in~\cite{MW} for proving shadowing of the dynamical system for orbits outside the separatrix in~\eqref{eqn:alphabeta} and large orbits can be shadowed for long times in~\eqref{eqn:nlsdwp-gK}!

Following the analysis in~\cite{MW}, we now write the system~\eqref{eqn:sys-lin} with the orbit ${\chi}_*$ given by an exact periodic orbit of the finite dimensional truncation:
 \begin{align*}
 \chi(t) &= \chi_*(t) + \eta,\\
  &\equiv \left(\tilde{A}(t)+ \eta_A(t), \tilde{\alpha}(t)+ \eta_\alpha(t), 
  \tilde{\beta}(t)+\eta_\beta(t) \right),
\end{align*}
where the $\eta_A$, $\eta_\alpha$ and $\eta_\beta$ will account for the error terms in gluing the finite dimensional model equations into the full infinite dimensional model.
This corresponds to a solution of the form:
\begin{equation*}
u(x,t) = e^{i \theta(t)} \left((\tilde{A}(t) + \eta_A(t) )\, \psi_0 + [(\tilde{\alpha} (t) + \eta_\alpha(t)) + i (\tilde{\beta} (t) + \eta_\beta(t))] \psi_1 + R(x,t)\right) 
\end{equation*}
with initial conditions
\begin{equation*}
u_0(x) = 
e^{i \theta(0) } \left(\, \tilde{A}(0)\, \psi_0(x) + [\tilde{\alpha}(0)+ i \tilde{\beta}(0)]\psi_1(x)\right).
\end{equation*}

Centered about ${\chi}_*$, equation~\eqref{eqn:nlsdwp-gK} becomes the system:
\begin{align*}
\dot{\vec{\eta}} & = D_{{\chi}} F_{\rm FD} \left({\chi}_*(t)\right) \vec{\eta} + [\vec{F}_{\rm FD} (\chi_{*} + \vec{\eta}) - \vec{F}_{\rm FD} (\chi_*)- D_{\chi} \vec{F}_{\rm FD} (\chi_*) \vec{\eta} ] + \vec{G}_{\rm FD} (\chi_{*},\vec{\eta};R,\bar{R})  \\
i R_t   & =  (H - \Omega_0) R + (\tilde{A}^2 + 3 \tilde{\alpha}^2 + \tilde{\beta}^2) R + F_b \left({\chi}_{*}+{\eta}\right) + F_R \left({\chi}_{*} + {\eta}; R, \bar{R}\right)  \\
& \; \;   + \left(\ A^2-\tilde{A}^2\ +\ 3(\alpha^2-\tilde{\alpha}^2)\ +\, \beta^2-\tilde{\beta}^2\right) R
\end{align*}
with
\begin{equation*}
\dot{\theta}  =  -\Omega_0 + A^2 + 3 \alpha^2 + \beta^2 + G_\theta (R, \bar{R}; A, \alpha, \beta).
\end{equation*}
Here, $\vec{F}_{\rm FD}$ and $F_b$ represent the finite dimensional and infinite dimensional projections of the terms dependent only upon $\chi (t)$ in the ansatz.  Also, $G_{\rm FD}$ and $G_\theta$ represent the interaction terms in the finite dimensional between $R$ and $\chi$, as given by the terms $\mathit{Error}_A$, $\mathit{Error}_\alpha$, $\mathit{Error}_\beta$, $\mathit{Error}_\theta$ from Appendix A of~\cite{MW}.  The term $F_R$ is the interaction term in the infinite dimensional evolution, as it is in Appendix A of~\cite{MW}. In particular, the evolution equations for $\vec \eta$ are similar to those for the finite dimensional linearization $(\delta \alpha, \delta \beta, \delta A)$ in~\eqref{eqn:sys-lin-alpha} but now with an additional forcing term given by $G_{\rm FD}$.  See Appendix $A$ of~\cite{MW} for full details.

The estimates
\begin{equation*}
\left|\, \vec{F}_{\rm FD} (\chi_{*} + \vec{\eta}) - \vec{F}_{\rm FD} (\chi_{*}) -  D_{\chi} F_{\rm FD} (\chi_*)\vec{\eta}\right| = 
\bigO{\tilde{A} |\vec{\eta}|^2 + |\vec{\eta}|^3\ }
\end{equation*}
are derived in~\cite{MW}. We decompose $R =\tilde{R} + w$, where $\tilde{R}$ the leading order part of $R$, driven
by the periodic solution $\chi_*(t)$ as defined in~\eqref{eqn:tildeR}, and $w$ is a correction.
Introduce, $\tilde{M}(t)$, a fundamental solution matrix for the system of ODEs with time-periodic coefficients:
\begin{equation}
\partial_t\eta = D_\chi F(\chi_*(t))\, \eta.
\nn\end{equation}

Then, we are able to study the following system of integral equations
 for $\eta(t), \ w(x,t), \, \theta (t)$:
\begin{align*}
\vec{\eta} (t) & = \int_0^t  \tilde{M}(t) \tilde{M}^{-1} (s) \\
& \hspace{1cm} \times \left[ \vec{F}_{\rm FD} (\chi_{*}+{\eta}) -  \vec{F}_{\rm FD} (\chi_{*}) - D_\chi \vec{F}_{\rm FD} (\chi_{*}) \vec{\eta}+ \vec{G}_{\rm FD} (\chi_{*},\vec{\eta};R,\bar{R}) \right] ds , \\
w(x,t) &= \int_0^t e^{iH(t-s)-i \Omega_0 (t-s) +i \int_s^t (\tilde{A}^2 + 3 \tilde{\alpha}^2 + \tilde{\beta}^2) (s') ds'} \\
& \hspace{1cm} \times P_c \left[ (F_b (\chi_{*} + \vec{\eta}) - F_b (\chi_{*})) + 
F_R(\chi_{*},\vec{\eta};R,\bar{R}) \right] ds  , \\
\theta &=  \theta_0 + \int_0^t \left[-\Omega_0 + A^2 + 3 \alpha^2 + \beta^2 + G_\theta (R, \bar{R}; \chi_{*},\vec{\eta}) \right] ds.
\end{align*}
We view a solution, $(\vec{\eta},w)$, of this system of integral equations  as fixed point of a mapping~$\mathcal{M}$:
 \begin{equation}
 (\vec{\eta},w) = \mathcal{M} (\vec{\eta},w).
\label{fixedpoint}
 \end{equation}

Given the bounds proven above in the proof of Theorem~\ref{thm:main-eq}, $\mathcal{M}$ is a contraction map in a particular Banach space in both $x$ and $t$ designed to optimally measure the dynamics of both $\eta$ and $w$. To that end, we recall the Strichartz space
\begin{equation*}
L^p_t W^{k,q}_x = L^p ([0,T^*] ; W^{k.q} (\mathbb{R})),
\end{equation*}
where $\frac{2}{p} + \frac{1}{q} = \frac12$ and $T^*$ is given in Theorem~\ref{thm:main-eq}.
Following~\cite{MW}, we define the  space 
\begin{align*}
X(I) & = X([0,T_*(\dd)]) \\
 & = \left\{\ (\vec{\eta},w): \eta\in L^\infty_t([0,T_*(\dd)]),\; w\in L^\infty_t([0,T_*(\dd)]; H^1_x) \cap L^4_t([0,T_*(\dd)]; L^\infty_x)
\right\}
\end{align*}
equipped with the natural norm
\begin{equation*}
\| (\vec{\eta},w) \|_{X(I)} = \| \vec{\eta} \|_{L^\infty_t(I)} + \| w \|_{L^4_t(I; L^\infty_x)} + \| w \|_{L^\infty_t(I; H^1_x)},
\end{equation*}
 \begin{equation}
{\rm where}\; \;  I=[0,T_*(\dd)].
 \nn\end{equation}
  We define $B_\dd (I) \subset X(I)$ such that $(\vec{\eta},R) \in B_\dd (I)$ if and only if
\begin{equation*}
\| (\vec{\eta},w) \|_{X(I)}  \le \dd^{\frac{1}{2}+\delta_1} ,
\end{equation*}
 where $\delta_1> 0$ must be chosen in the course of the analysis. 
 
Then the following proposition makes the desired contraction mapping precise.
\begin{prop} 
  The mapping $\mathcal{M} : X(I) \to X(I)$, defined in~\eqref{fixedpoint},  
has the properties
\begin{enumerate}
\item $\mathcal{M}: B_\dd(I) \to B_\dd(I)$.
\item There exists $\kappa<1$ such that given
$(\vec{\eta}_j,w_j) \in  B_\dd(I)$ for $j = 1,2$, 
\begin{equation*}
d(\mathcal{M}(\vec{\eta}_1,w_1),\mathcal{M}(\vec{\eta}_2,w_2)) \leq\, \kappa\  d((\vec{\eta}_1,w_1) , (\vec{\eta}_2,w_2)).
\end{equation*}
\end{enumerate} 
Thus, there exists a unique solution $(\vec{\eta},w)$ in 
$B_\dd(I)$.
\end{prop}

The main result then follows by applying the asymptotic analysis from Section~$5$ of~\cite{MW} to prove for example bounds of the form
\begin{align*}
\| \vec{\eta} \|_{L^\infty_t}  = & \Big{\lVert} \int_0^t M (t) M^{-1}(s) \big[ \vec{F}_{\rm FD} (\chi_{*} + \vec{\eta}) -  \vec{F}_{\rm FD} (\chi_{*}) - D_{\chi^*} \vec{F}_{\rm FD} (\chi_{*}) \vec{\eta}  \\
&  + \vec{G}_{\rm FD} (\chi_{*} + \vec{\eta};R,\bar{R}) \big]  ds \Big{\rVert}_{L^\infty}
\end{align*}
and
\begin{align*}
\norm{w}_{L^\infty_t H^1_x \cap L^4_t L^\infty_x} & =  \Big{\lVert} \int_0^t e^{iH(t-s)}  e^{i \Omega_0 (t-s)} e^{i \int_s^{t} (\tilde{A}^2 + 3 \tilde{\alpha}^2 + \tilde{\beta}^2) (s') ds'} \\
& \times P_c \left[ (F_b (\chi_{*}+\vec{\eta}) - F_b (\chi_{*})) + F_R (\chi_{*},\vec{\eta};R,\bar{R}) \right] ds \Big{\rVert}_{L^\infty H^1 \cap L^4 L^\infty} ,
\end{align*}
the proofs of which rely heavily upon the bounds presented in Sections~\ref{sec:peramp} and~\ref{sec:fundsol} above.

\section{Discussion, Further Problems, and a Remark}
\label{sec:discussion}
We consider NLS/GP with a symmetric double well potential having sufficiently separated wells.
We have shown that there are periodic solutions of NLS/GP which, on a long time scale, shadow  periodic orbits 
of the finite dimensional ODE model both inside and outside the separatrix. An obvious question is whether similar behavior can be shown in systems with multiple potential wells. It is shown in~\cite{Goodman:2011,Goodman:2014} that periodic orbits and a class of quasi-periodic orbits arise as Hamiltonian Hopf bifurcation.  Work is underway to show that these periodic orbits are shadowed in the full PDE model.   


\begin{rem}
The extension of the proof of~\cite{MW} to the more general class of periodic orbits considered here was facilitated by a closed-form solution of the reduced system.  It is instructive to consider  how this was used:
\begin{enumerate}
\item Bounds on the solution's period are needed in order to obtain bounds on a Duhamel operator in the proof of~\cite{MW}, Page $21$, Equation $(4.4)$. The necessary bounds follow from the explicit solution.  In the case where there is no explicit solution, for trajectories near the separatrix the dominant contribution to the period can be computed by an asymptotic analysis near the saddle fixed point at the origin.
\item In Lemma~\ref{lem:eigen}, the symmetries of the exact solution are used to show that the monodromy matrix has an eigenvalue $\lambda=1$ of algebraic multiplicity three and geometric multiplicity two. This information is obtainable from the symmetry of the equation and does not require the exact solution.
\item The exact solution simplifies the estimate of the bounds of the various terms in the fundamental solution operator, equations~\eqref{eqn:percoeffs}-\eqref{eqn:tcoeffs}. In the absence of exact formulas, these bounds require terms beyond first-order in the expansion of $M(t)$. An alternate way to determine them is to first put the Hamiltonian system~\eqref{eqn:alphabeta} into normal form, which agrees the Hamiltonian of the Duffing oscillator to leading order. This is demonstrated in the Appendix using the method of Lie Transforms.  Truncation in the normal form transformation introduces a secondary source of error that must be controlled carefully in any estimates.
\end{enumerate}
\end{rem}

\section*{Acknowledgments} RHG received support from NSF DMS--0807284.  JLM was supported by U.S. NSF Grant DMS--1312874, an IBM Junior Faculty Development Award from UNC and a Guest Professorship from Universit\"at Bielefeld.  MIW was supported in part by U.S. NSF Grant DMS--10-08855, and the Columbia Optics and Quantum Electronics IGERT NSF grant DGE--1069420.  The authors also wish to thank the anonymous reviewers for a careful reading of the draft and many useful suggestions that improved the exposition greatly.

\appendix

\section{Outline of approach using Hamiltonian normal forms}
\label{a:lie}

A key ingredient in the proof of shadowing is prove bounds on the fundamental solution matrix operator
for a reduced system. While these can be studied via the explicit solutions of the reduced system arising for double wells, more generally explicit solutions are unavailable and therefore a different approach must be taken. For example, in the case where $V(x)$ is a triple well potential, the orbits can only be calculated via perturbative methods~\cite{Goodman:2011,Goodman:2014}. In this appendix we give a brief sketch of the general normal form method and the results of its implementation for~\eqref{eqn:alphabeta}, in the regime where $\sigma$ is a small parameter; see~\cite{Meyer:2010} for details.

Consider a general Hamiltonian of the form 
\begin{equation}
H(Q,P) = H_0(Q,P) + \sigma f(Q,P;\sigma),
\label{H0f}
\end{equation}
where $\sigma\ll1$, $H_0$ is quadratic (so that the associated differential equations are linear and thus integrable) and $f$ has a finite or convergent power series. A normal form for the Hamiltonian~\eqref{H0f} is an equivalent but simpler Hamiltonian obtainable by a suitable canonical change of variables $(Q,P) = \Phi(q,p) = (q,p) + \ldots$ under which new Hamiltonian $K(q,p) = H(\Phi(q,p)) = H_0(q,p) + \sigma g(q,p;\sigma)$. Roughly, the function $g(q,p;\sigma)$ should contain as few terms as possible at each order in perturbation theory. It is well-known~\cite[\S10.4]{Meyer:2010}  that for a Hamiltonian like~\eqref{HDW} with leading-order part
$H_0 = \tfrac{\W_{10}}{2}P^2$,
the normal form Hamiltonian is
$$
K(q,p) = \frac{\W_{10}}{2}p^2 + \sigma g(q;\sigma).
$$
That is, the normal form equation is separable, with $p$-dependent kinetic and $q$-dependent potential energy.

Using perturbative technique based on the \emph{Lie transform}, we can calculate the change of coordinates and the transformed Hamiltonian to arbitrary order using Mathematica:
\begin{align*}
\alpha = Q & =  \left(
1
+\frac{q^2}{3 \Omega_{10}}
+\frac{q^4}{30 \Omega_{10}^2}
+\frac{q^6}{630 \Omega_{10}^3}
+ O{\left(q^8\right)}
\right) q; 
\\
\beta = P & =  \left(
1
-\frac{q^2}{\Omega_{10}}
+\frac{5 q^4}{6 \Omega_{10}^2}
-\frac{61 q^6}{90 \Omega_{10}^3}
+ O{\left(q^8\right)}
\right)p; \\
K(q,p) & = 
\frac{\Omega_{10}}{2}p^2  
- \frac{\sigma }{2}q^2  
+ q^4
+\frac{- \sigma q^4 + 4 q^6}{3 \Omega_{10}}
+\frac{4 \left(-q^6 \sigma + 9 q^8 \right)}{45 \Omega_{10}^2} + 
O{\left(\sigma q^8,q^{10}\right)}.
\end{align*}

The three leading terms of $K(q,p)$ are equivalent to the 
Duffing Hamiltonian~\eqref{eqn:duffing}, which has well-known solutions involving Jacobi elliptic functions~\cite{NIST:DLMF}.  These solutions then form the leading order parts of Poincar\'e-Lindstedt approximations to true periodic orbits of $K(q,p)$, which can be computed to arbitrary order using computer algebra.

While the asymptotic series defining a normal form change of variables does not generally converge, there exist, in some cases, useful error estimates for finite truncations of the  series, for example Giorgilli and Galgani~\cite{Giorgilli:1985vq}, who prove an estimate of \emph{Nekhoroshev} type~\cite{Nekhoroshev:1971}. 
By truncating the series to order $r$, they show that the error due to the normal form approximation can be controlled to within $\bigO{\sigma}$ for times of $\bigO{\sigma^{-r}}$. Further, as $\sigma$, the order $r$ can be chosen in order to control the error over \emph{exponentially} long time scales in $1/\sigma$.
  Using the normal form approximation introduces approximations into the fundamental solution operator of the finite-dimensional system, which must be controlled in order to verify the bounds~\eqref{Tdelta} and~\eqref{Mbound} which are discussed in Section~\ref{sec:rescaled}. The time scales on which these estimates are proven must be reconciled with the time scales on which the normal form approximation is valid.

\bibliographystyle{firstinitials} 
\bibliography{gmw-nls}

\end{document}